\documentclass[a4paper,11pt,english]{smfart}
\usepackage{amsfonts}
\usepackage{amsmath} 
\usepackage{hyperref} 
\usepackage{latexsym}
\usepackage{array}
\usepackage{amssymb}
\usepackage{enumerate}
\usepackage{smfthm}
\theoremstyle{plain}

\newtheorem*{acknowledgements}{Acknowledgements}
\newtheorem{assumption}{Assumption}

\renewcommand{\(}{  \big(   }
\renewcommand{\)}{  \big)   }
\newcommand{\R}{  \mathbb{R}   }

\newcommand{\eps}{\varepsilon}
\newcommand{\e}{  \text{e}   }

\newcommand{\C}{  \mathbb{C}   }
\newcommand{\Z}{  \mathbb{Z}   }
\newcommand{\N}{  \mathbb{N}   }

\newcommand{\J}{  \mathcal{J}   }

\renewcommand{\H}{  \mathcal{H}   }

\newcommand{\T}{  \mathbb{T}   }

\newcommand{\dis}{\displaystyle}

\newcommand{\om}{  \omega   }
\newcommand{\ov}{  \overline  }
\renewcommand{\a}{  \alpha   }
\renewcommand{\b}{  \beta   }
\newcommand{\s}{  \sigma   }
\newcommand{\G}{  \Gamma   }
\newcommand{\z}{  \overline{z}  }

\renewcommand{\phi}{  \varphi  }
\renewcommand{\L}{  \mathcal{L}   }
\newcommand{\wh}{  \widehat   }
\newcommand{\<}{  \langle   }
\newcommand{\cp}[1]{  \big\{\,{ #1 }\, \big\}  }
\renewcommand{\>}{  \rangle   }
\numberwithin{equation}{section}

 \author{ Beno\^it Gr\'ebert}
\address{Laboratoire de Math\'ematiques J. Leray, Universit\'e de Nantes, UMR CNRS 6629\\
2, rue de la Houssini\`ere \\
44322 Nantes Cedex 03, France.}
\email{benoit.grebert@univ-nantes.fr}

\author{ Laurent Thomann }
\address{Laboratoire de Math\'ematiques J. Leray, Universit\'e de Nantes, UMR CNRS 6629\\
2, rue de la Houssini\`ere \\
44322 Nantes Cedex 03, France.}
\email{laurent.thomann@univ-nantes.fr}

\title[KAM for  the  quantum harmonic oscillator]
{KAM for the  quantum harmonic oscillator }

\begin{document}

\begin{abstract}
In this paper we prove an abstract KAM theorem for infinite dimensional Hamiltonians systems. This result extends  previous works of S.B. Kuksin and J. P\"oschel  and uses recent techniques of H. Eliasson and S.B. Kuksin. As an application we show that some 1D nonlinear Schr\"odinger equations with harmonic potential admits many quasi-periodic solutions.  In a second application we prove the reducibility of the   1D Schr\"odinger equations with the harmonic potential and a quasi periodic in time potential.
 \end{abstract}

\thanks{
The first author was supported in part by the  grant ANR-06-BLAN-0063.\\
The second author was supported in part by the  grant ANR-07-BLAN-0250.}

\maketitle
\tableofcontents

\section{Introduction}
 
 Let $\Psi : \N\longrightarrow [0,+\infty[$ so that  $\Psi(j)\geq j$ for all $j\geq 1$. We consider 
 the (complex) Hilbert space 
 ${\ell}^{2}_{\Psi}$ defined by the norm
 \begin{equation*}
 \|w\|_{\Psi}^{2}=\sum_{j\geq 1}|w_{j}|^{2}\Psi^{2}(j).
  \end{equation*}
  We define the symplectic phase space $\mathcal{P}=\mathcal{P}^{\Psi}$ as 
    \begin{equation}\label{def.esp}
    \mathcal{P}=\T^{n}\times \R^{n}\times {\ell}^{2}_{\Psi} \times {\ell}^{2}_{\Psi},
  \end{equation}
  equipped  with the canonic symplectic structure: $$\sum_{j=1}^n d\theta_j\wedge dy_j\ +\ \sum_{j\geq 1} du_j\wedge dv_j.$$
  For $(\theta,y,u,v) \in \mathcal{P}$ we introduce the following Hamiltonian in normal form
 \begin{equation}\label{Ham.N}
N=\sum_{j=1}^{n}\om_{j}(\xi)y_{j}+\frac12\sum_{j\geq 1}\Omega_{j}(\xi)(u_{j}^{2}+v_{j}^{2}),
\end{equation}
where $\xi \in \R^{n}$ is an external parameter.\\[2pt]
 In \cite{Kuk1},  (see also \cite{Kuk2} and a slightly generalised version in \cite{Poschel}) S.B. Kuksin has shown  the persistence of $n-$dimensional tori for the perturbed Hamiltonians $H=N+P$ with general conditions on the frequencies $\om_{j}, \Omega_{j}$ and perturbation $P$ which essentially are the following :  Firstly the frequencies  satisfy some Melnikov conditions and the external frequencies $\Omega_{j}$ have to be well separated in the sense that there exists $d\geq 1$  so that roughly speaking (see Assumption \ref{AS2} below)
 \begin{equation}\label{freq}
 \Omega_{j}(\xi)\approx j^{d}.
 \end{equation}
Denote by $\mathcal{P}^{a,p}$ the phase space given by the weight $\Psi(j)= j^{p/2}\e^{aj}$ where $p\geq 0$ and $a\geq 0$.  Secondly, the perturbation is real analytic and the corresponding Hamiltonian vector field is so that 
 \begin{equation}\label{XP}
X_{P}:\mathcal{P}^{a,p}\longrightarrow \mathcal{P}^{a,\ov{p}} \quad 
\text{with} \quad \left\{\begin{array}{ll} 
\quad  \ov{p}\geq p &\text{for} \quad d>1, \\[6pt]  
\quad \ov{p}> p &\text{for} \quad d=1, \end{array} \right.
  \end{equation}
  where $d$ is the constant which appears in \eqref{freq}.
For instance, the Schr\"odinger and the wave equation on $[0,\pi]$ with Dirichlet boundary conditions satisfy the previous conditions, see respectively the KAM results of Kuksin-P\"oschel  \cite{KukPosch} and P\"oschel \cite{Poschel2}. Indeed the result in \cite{KukPosch} is stronger because there is no external parameter $\xi$ in  the equation.\\[3pt]
Now, if we consider the nonlinear harmonic oscillator
 \begin{equation}\label{NLS}
i\partial_t u=-\partial^{2}_{x} u +x^{2}u  +V(x)u+ |u|^{2m}u ,\quad
(t,x)\in\R\times {\R},
\end{equation}
with real and bounded potential $V$, we have $\Omega_{j}\sim 2j+1$, hence $d=1$ but the Hamiltonian perturbation which is here 
 \begin{equation}\label{pert}
P=\int_{\R}(u\bar u)^{m+1}\text{d}x,
 \end{equation}
does not satisfy the strict smoothing condition \eqref{XP} (see Section \ref{Sect.5} for more details).  \\

The aim of  this paper is to prove a KAM theorem (Theorem \ref{thmKAM}) in the case $d=1$ and $p=\ov{p}$ in \eqref{freq} and \eqref{XP}. To compensate the lack of smoothing effect of $X_{P}$ we need some additional conditions (see Assumption \ref{AS4}) on the decay of the $P$ derivatives  (in the spirit of the so-called T\"oplitz-Lipschitz condition used by Eliasson \& Kuksin in \cite{ElKuk1}) which will be satisfied by the perturbation  \eqref{pert}. The general strategy is explained with more details in Section \ref{Strategy}.\\
Notice that S.B. Kuksin  has already considered in \cite{Kuk2} the  harmonic oscillator with a smoothing nonlinearity of type $\dis P=\int_{\R}\phi(|u\star \xi|)\text{d}x$ where $\xi$ is a fixed smooth function.\\[2pt]

We present two applications of our abstract result concerning the harmonic oscillator $T=-\partial^{2}_{x}+x^{2}$. Let $p\geq 2$ and denote by $\ell^{2}_{p}$ the space $\ell^{2}_{\Psi}$ with $\Psi(j)=j^{p/2}$. The operator $T$ has    eigenfunctions $(h_{j})_{j\geq 1}$ (the Hermite functions) which satisfy $Th_{j}=(2j-1)h_{j}, \;j\geq 1$ and form a  Hilbertian basis of  $L^{2}(\R)$. Let $u=\sum_{j\geq1}u_{j}h_{j}$ be a typical element of $L^{2}(\R)$. Then  $(u_{j})_{j\geq 1}\in \ell^{2}_{p}$ if and only if $u\in \H^{p}:= D(T^{p/2})=\{u\in L^{2}(\R)\mid T^{p/2}u\in L^{2}(\R)\}$. Indeed $\H^{p}$ is a Sobolev space based on $T$ and we can check that 
$$\H^{p}=D(T^{p/2})=\{u\in L^{2}(\R)\mid  x^\alpha\partial^\beta u \in L^2(\R) \text{ for } \alpha+\beta\leq p\}.$$
In this context, we are able to apply our KAM result to \eqref{NLS} and  we obtain (see Theorem \ref{main1}
 for a more precise statement)

 \begin{theo}\label{theo:NLS}
Let $m\geq 1$ be an integer. For typical potential $V$ and for $\epsilon>0$ small enough, the nonlinear Schr\"odinger equation 
\begin{equation}\label{NLS:theo}
i\partial_t u=-\partial^{2}_{x} u +x^{2}u  +\eps V(x)u\pm \epsilon |u|^{2m}u 
\end{equation}
 has many quasi-periodic solutions in $\H^{\infty}$.
\end{theo}

 Here the notion of ``typical potential'' is vague. This means that there exists rather a large class of perturbations of the harmonic oscillator so that the result of Theorem \ref{theo:NLS} holds true (unfortunately our result does not cover the case $V=0$). Since the definition of this class is technical, we postpone it to  Section \ref{Sect.5}. \\
 
 The physical  motivation for considering equation \eqref{NLS:theo} (for $V=0$) comes from the Gross-Pitaevski equation used in the study of Bose-Einstein condensation (see \cite{GP}). The harmonic potential $x^{2}$ arises from  a Taylor expansion near the bottom of general smooth well. In our work, we have to add a small linear perturbation $V$ to the harmonic potential in order to avoid resonances (see the non resonance condition \eqref{dio} below).\\

The generalisation of such a result in a multidimensional setting is not evident for a spectral reason:   the spectrum of the linear part is no more well separated. We could expect to adapt the tools introduced in \cite{ElKuk1} but the arithmetic properties of the corresponding spectra are not the same: in \cite{ElKuk1} the free frequencies are $j_1^2+j_2^2+\cdots+j_D^2$ for all $j_1,\cdots, j_D\in \Z$, while in our case they are $2(j_1+j_2+\cdots+j_D)+D$ for all $j_1,\cdots, j_D\in \N$. Nevertheless we mention that it is still possible to obtain a Birkhoff normal form for \eqref{NLS} as recently proved in \cite{GIP}.\\
A  consequence  of Theorem \ref{theo:NLS} is the existence of periodic solutions to \eqref{NLS:theo}.  There are other approaches to construct  periodic solutions of this equation. For instance, the gain of compacity yielded by the confining potential $x^{2}$ allows the use of variational methods. We develop this point of view in the appendix.\\

The second application concerns the reducibility of a linear harmonic oscillator, $T=-\partial^{2}_{x}+x^{2}$,   on $L^2(\R)$ perturbed by a quasi periodic in time potential. Such kind of reducibility result for PDE using KAM machinery was first obtained by Bambusi \& Graffi   \cite{BamGra}  for Schr\"odinger equation with an $x^\beta$ potential, $\beta$ being strictly larger than 2 (notice that in that case the exponent  $d>1$ in the asymptotic of the frequencies \eqref{freq}). This result was recently extended by Liu and Yuan \cite{LiuYuan} to include the Duffing oscillator.\\  Here we follow the more recent approach developed by  Eliasson \& Kuksin (see \cite{ElKuk2}) for the Schr\"odinger equation on the multidimensional torus. Namely we consider the linear equation
\begin{equation*}
i\partial_t u=-\partial^{2}_{x} u +x^2u +\epsilon V(t\omega,x)u, \quad u=u(t,x),\ x\in \R,
\end{equation*}
where $\epsilon >0$ is a small parameter and the frequency vector $\omega$ of forced oscillations is regarded as a parameter in $\mathcal U \subset \R^n$. We  assume  that the potential $V: \ \T^n\times \R\ni(\theta, x)\mapsto \R$ is analytic in $\theta$ on $|\text{Im}\,\theta|<s$ for some $s>0$, and $\mathcal{C}^{2}$ in $x$, and we suppose that there exists $\delta>0$ and $C>0$ so that for all $\theta \in [0,2\pi)^{n}$ and $x\in \R$ 
\begin{equation}\label{*}
 |V(\theta,x)|\leq C(1+x^{2})^{-\delta},\quad\; |\partial_{x}V(\theta,x)|\leq C,\quad\; |\partial_{xx}V(\theta,x)|\leq C. 
 \end{equation}
 In Section \ref{Sect.6} we consider the previous equation as a linear non-autonomous equation in the complex Hilbert space $L^2(\R)$ and we prove (see Theorem \ref{thmKAM2} for a more precise statement)
 \begin{theo}\label{theo:LS} Assume that $V$ satisfies \eqref{*}. Then there exists $\epsilon_0$ such that for all $0\leq\epsilon<\epsilon_0$ there exists $\Lambda_{\eps} \subset [0,2\pi)^n$ of positive measure and asymptotically full measure: $\mbox{Meas}(\Lambda_{\eps} ) \to (2\pi)^n$ as $\epsilon \to 0$, such that for all 
$\omega\in \Lambda_{\eps} $, the linear Schr\"odinger equation 
\begin{equation}\label{LS}
 i\partial_t u=-\partial^{2}_{x} u +x^2u +\epsilon V(t\omega,x)u
 \end{equation}
 reduces, in $L^2(\R)$, to a  linear equation with constant coefficients (with respect to the time variable).
  \end{theo}
  
   In particular, we prove the following result concerning the solutions of \eqref{LS}. 
  \begin{coro}\label{Coro1.3}
  Assume that $V$ is  $\mathcal{C}^{\infty}$ in $x$ with all its derivatives bounded  and satisfying  \eqref{*}. Let $p\geq 0$ and $u_{0}\in \H^{p}$. Then there exists $\eps_{0}>0$ so that for all $0<\eps<\eps_{0}$ and $\om \in \Lambda_{\eps}$,    there exists a unique solution $u \in \mathcal{C}\big(\R\,;\,\H^{p}\big)$ of \eqref{LS} so that $u(0)=u_{0}$. Moreover, $u$ is almost-periodic in time and we have the bounds 
  \begin{equation*}
  (1-\eps C)\|u_{0}\|_{\H^{p}}\leq \|u(t)\|_{\H^{p}}\leq  (1+\eps C)\|u_{0}\|_{\H^{p}}, \quad \forall \,t\in \R,
  \end{equation*}
  for some $C=C(p,\om)$.
  \end{coro}
  
   \begin{rema}
  In the very particular case where $V$ satisfies \eqref{*} and is independent of $\theta$, the result of Corollary \ref{Coro1.3} is easy to prove. In that case,  the solution of \eqref{LS} reads 
  $$u(t,x)=\sum_{n\geq 0}c_{n}\e^{-i\lambda^{2}_{n}t}\phi_{n}(x),$$
  where $(\phi_{n})_{n\geq 0}$ and $(\lambda_{n})_{n\geq 0}$ are the eigenfunctions and the eigenvalues of  $-\partial_{x}^{2}+x^{2}+\eps V(x)$, and some $(c_{n})_{n\geq 0}\in \C$. The result follows thanks to the asymptotics of $\phi_{j}$ when $\eps \to 0$   (see Section \ref{Sect.5} for similar considerations.)
  \end{rema}

  The previous results show that all solutions to \eqref{LS} remain bounded in time, for a large set of parameters $\om\in  [0,2\pi)^{n}$.  A natural question is whether we can find a real valued potential $V$, quasi-periodic in time and a solution $u\in \H^{p}$ so that $\|u(t)\|_{\H^{p}}$ does not remain bounded when $t\longrightarrow +\infty$. J.-M. Delort \cite{Delort} has recently shown  that this is the case if $V$ is replaced by  a pseudo differential operator : he proves that there exist smooth solutions so that for all $p\geq 0$ and $t\geq 0$, $\|u(t)\|_{\H^{p}}\geq c t^{p/2}$, which is the   optimal growth. We also refer to the introduction of \cite{Delort} for a survey on the  problem of Sobolev growth for the linear Schr\"odinger equation.\\
   
Another way to understand the result of Theorem \ref{theo:LS} is in term of Floquet operator (see \cite{Eli} and \cite{Wang} for mathematical considerations, and \cite{EV, KY} for the physical meaning).  Consider on $L^2(\R)\otimes L^2(\T^n)$ the Floquet Hamiltonian  
 \begin{equation}\label{Floq}
 K:=i\sum_{k=1}^n\omega_k \frac{\partial}{\partial \theta_k} -\partial^{2}_{x}  +x^2 +\epsilon V(\theta,x),
\end{equation}
then we have
 \begin{coro}
 Assume that $V$ satisfies  \eqref{*}. There exists $\eps_{0}>0$ so that for all $0<\eps<\eps_{0}$ and $\om \in \Lambda_{\eps}$, the spectrum of the Floquet operator $K$ is pure point.
 \end{coro}
 A similar result, using a different KAM strategy,  was obtained by W.M. Wang in \cite{Wang} in the case where 
 $$V(t\omega,x)=|h_1(x)|^2\sum_{k=1}^n \cos(\omega_k t +\phi_k)$$
where $h_1$ is the first Hermite function. \\

At the end of Section \ref{Sect.6} we make explicit computations in the case of a potential which is independent of the space variable. This example shows that one can not avoid to restrict the choice of parameters $\om$ to a Cantor type set in Theorem \ref{theo:LS}.

\begin{acknowledgements}
The first author thanks Hakan Eliasson and Serguei Kuksin for helpful suggestions at the principle of this work.
Both authors thank Didier Robert for many clarifications in spectral theory.
\end{acknowledgements}

 \section{Statement of the abstract result} ~

We give in this section our abstract KAM result.\\

\subsection{The assumptions on the Hamiltonian and its  perturbation}~\\[5pt]
Let $\Pi\in \R^{n}$ be a bounded closed set so that $\text{Meas}(\Pi)>0$, where $\text{Meas}$ denote the Lebesgue measure in $\R^{n}$. The set $\Pi$ is the space of the external parameters $\xi$. 
  Denote by $\Delta_{\xi\eta}$ the difference operator in the variable $\xi$ : 
  \begin{equation*}
  \Delta_{\xi\eta}f=f(\cdot,\xi)-f(\cdot,\eta).
  \end{equation*}
  For $\l =(l_{1},\dots, l_{k},\dots)\in \Z^{\infty}$ so that only a finite number of coordinates are non zero, we denote by $\dis |l|=\sum_{j=1}^{\infty}|l_{j}|$ its length, and $\dis \<l\>=1+|\sum_{j=1}^{\infty}jl_{j}|$. We set $\dis \mathcal{Z}=\{(k,l)\neq 0,\;|l|\leq 2\}\subset \Z^{n}\times \Z^{\infty}.$\\
  
  The first two assumptions we make,  concern the frequencies of the Hamiltonian in normal form \eqref{Ham.N}
    \begin{assumption}[Nondegeneracy]\label{AS1} 
    Denote by $\om=(\om_{1},\dots,\om_{n})$ the internal frequencies. We assume that the map $\xi \mapsto \om(\xi)$ is an homeomorphism from $\Pi$ to its image which is Lipschitz continuous and its inverse also. \\
    Moreover we assume that for all $(k,l)\in \mathcal{Z}$ 
     \begin{equation}\label{Non.dg}
  \text{Meas}\Big(\big\{\,\xi\;:\:k\cdot \om(\xi)+l\cdot \Omega(\xi)=0\,\big\}\Big) =0,
  \end{equation}
 and  for all $\xi \in \Pi$
$$l\cdot \Omega(\xi)\neq 0, \quad \forall\, 1\leq |l|\leq 2.$$
\end{assumption}
 \begin{assumption}[Spectral asymptotics]\label{AS2} 
Set  $\Omega_{0}=0$. We assume that there exists $m>0$ so that for all $i,j\geq 0$ and uniformly on $\Pi$
 \begin{equation*}
 | \Omega_{i}-\Omega_{j}|\geq m|i-j|.
  \end{equation*}
  Moreover we assume that there exists $\beta>0$ such that the functions 
   \begin{equation*}
  \xi \longmapsto j^{2\beta} \Omega_{j}(\xi),
  \end{equation*}
   are uniformly Lipschitz on $\Pi$ for $j\geq 1$.
     \end{assumption}
     
       If the previous assumptions are satisfied (and actually without assuming \eqref{Non.dg}), J. P\"oschel \cite{Poschel} proves that there exist a finite set $\mathcal{X}\subset \mathcal{Z}$ and $\widetilde{\Pi}_{\a}\subset \Pi$ with $\text{Meas}(\Pi\backslash \widetilde{\Pi}_{\a})\longrightarrow 0$ when $\a \longrightarrow 0$, such that for all $\xi \in  \widetilde{\Pi}_{\a}$
        \begin{equation} \label{2.2bis}
 \big|k\cdot \om(\xi) +l\cdot \Omega(\xi) \big|\geq \alpha\frac{ \<l\>}{1+|k|^{\tau}},\quad (k,l)\in \mathcal{Z}\backslash \mathcal{X},
 \end{equation}
 for some large $\tau$ depending on $n$ and $\b$.\\
Then assuming \eqref{Non.dg}, J. P\"oschel proves \cite[Corollary C and its proof]{Poschel} that the non resonance condition \eqref{2.2bis} remains valid 
on all $\mathcal{Z}$, i.e 
        \begin{equation}\label{dio} 
 \big|k\cdot \om(\xi) +l\cdot \Omega(\xi) \big|\geq \alpha\frac{ \<l\>}{1+|k|^{\tau}},\quad (k,l)\in \mathcal{Z},\,\xi \in \widetilde{\Pi}_{\a}.
 \end{equation}

    In the sequel, we will use the    distance 
      \begin{equation*}
  |\Omega-\Omega'|_{2\beta,\Pi}=\sup_{\xi \in \Pi}\,\sup_{j\geq 1}j^{2\beta}\,|\Omega_{j}(\xi)-\Omega'_{j}(\xi)|
  \end{equation*}
and the semi-norm
   \begin{equation*}
  |\Omega|^{\L}_{2\beta,\Pi}=\sup_{\substack{ \xi,\eta \in \Pi\\ \xi\neq \eta}}\,\sup_{j\geq 1}j^{2\beta}\,\frac{|\Delta_{\xi\eta}\,\Omega_{j}|}{|\xi-\eta|}.
  \end{equation*}
    Finally, we set 
  $$   |\omega|^{\L}_{\Pi}+ |\Omega|^{\L}_{2\beta,\Pi}=M,$$
  where  $\dis  |\omega|^{\L}_{\Pi}=\sup_{\substack{ \xi,\eta \in \Pi\\ \xi\neq \eta}}\,\max_{1\leq k\leq n}\,\frac{|\Delta_{\xi\eta}\,\omega_{k}|}{|\xi-\eta|}.$
\begin{rema}   The proof of \eqref{dio} crucially uses the control of the Lipschitz semi-norm  $|\Omega|^{\L}_{2\beta,\Pi}$ (see \cite[Lemma 5]{Poschel}). For this reason in assumptions 3 and 4 below we have to control the Lipschitz version of each semi-norms introduced on $P$ or $X_P$.
\end{rema} 
Recall that the phase space $\mathcal{P}$ is defined by \eqref{def.esp}, with a weight $\Psi$ so that $\Psi(j)\geq j$, as in the beginning of the introduction.
 As in \cite{Poschel}, for $s,r>0$ we define the (complex) neighbourhood of $\T^{n}\times \big\{0,0,0\big\}$ in $\mathcal{P}$.
 \begin{equation}\label{def.dsr}
{D}(s,r)=
  \big\{(\theta,y,u,v)\in \mathcal{P}^{}\;s.t.\ |\text{Im}\, \theta|<s, |y|<r^{2}, \|u\|_{\Psi}+\|v\|_{\Psi}<r\big\}.
  \end{equation}
Let $r>0$. Then  for $W=(X,Y,U,V)$ we define
\begin{equation*}
|W|_{r}=|X|+\frac{1}{r^{2}}|Y|+\frac{1}r\big(\,\|U\|_{\Psi}+\|V\|_{\Psi}\big).
\end{equation*}
The next assumption concerns the regularity of the vector field associated to $P$. Denote by 
$$\dis X_{P}=(\,\partial_{y}P,\, -\partial_{\theta}P,\,\partial_{v}P,\,-\partial_{u}P\,).$$
Then

  \begin{assumption}[Regularity]\label{AS3} 
  We assume that there exist $s,r>0$ so that
  \begin{equation*}
  X_{P}\,:D(s,r)\times \Pi \longrightarrow \mathcal{P}^{}.
  \end{equation*}
  Moreover we assume that for all $\xi \in \Pi$, $X_{P}(\cdot,\xi)$ is analytic in $D(s,r)$ and that for all $w\in D(s,r)$,  $P(w,\cdot)$ and $X_{P}(w,\cdot)$ are Lipschitz continuous on $\Pi$. 
    \end{assumption}  
 \noindent   We then define the norms 
     \begin{equation*}
\|P\|_{D(s,r)}:=\sup_{D(s,r)\times \Pi}|P|<+\infty,
    \end{equation*}
    and
    \begin{equation*}
\|P\|^{\L}_{D(s,r)}= \sup_{\substack{\xi,\eta\in\Pi\\ \xi\neq \eta}} \,\sup_{D(s,r)}\,\frac{|\Delta_{\xi\eta}\,P|}{|\xi-\eta|}, 
\end{equation*}
where $\Delta_{\xi\eta}\,P=P(\cdot,\xi)-P(\cdot,\eta)$ and we define the semi-norms
  \begin{equation*}
\|X_{P}\|_{r,D(s,r)}:=\sup_{D(s,r)\times \Pi}|X_{P}|_{r}<+\infty,
    \end{equation*}
    and 
       \begin{equation*}
  \|X_{P}\|_{r,D(s,r)}^{\L}:=\sup_{\substack{\xi,\eta\in\Pi\\ \xi\neq \eta}} \,\sup_{D(s,r)}\,\frac{|\Delta_{\xi\eta}X_{P}|_{r}}{|\xi-\eta|}<+\infty. 
    \end{equation*}
where $\Delta_{\xi\eta}X_{P}=X_{P}(\cdot,\xi)-X_{P}(\cdot,\eta)$.\\
 In the sequel, we will often work in the complex coordinates
 \begin{equation*}
 z=\frac{1}{\sqrt{2}}(u-iv),\quad  \ov{z}=\frac{1}{\sqrt{2}}(u+iv).
 \end{equation*}
 Notice that this is not a canonical change of variables and in the variables $(\theta,y,z,\bar z)\in  \mathcal{P}$ the symplectic structure reads
 $$
 \sum_{j=1}^n d\theta_j\wedge dy_j\ +\ i\sum_{j\geq 1} dz_j\wedge d\bar z_j,
 $$
 and the Hamiltonian in normal form is 
 \begin{equation}\label{Ham.Comp}
N=\sum_{j=1}^{n}\om_{j}(\xi)y_{j}+\sum_{j\geq 1}\Omega_{j}(\xi)z_{j}\ov{z}_{j}.
 \end{equation}
 As we mentioned previously we need some decay on the derivatives of $P$. We first  introduce the space $\G^{\b}_{r,D(s,r)}$:
 Let $\b>0$, we say that $P\in  \G^{\b}_{r,D(s,r)}$ if $\<P\>_{r,D(s,r)}+\<P\>^{\L}_{r,D(s,r)} <\infty $ where : \\
$\bullet$ The norm $\<\,\cdot\,\>_{r,D(s,r)}$ is defined by the conditions \footnote{This means that $\< P\>_{r,D(s,r)}$ is the smallest real number which satisfies the mentioned conditions : this defines a norm.   }
    \begin{eqnarray*}
  \big\|P\big\|_{D(s,r)}&\leq & r^{2}\<P\>_{r,D(s,r)},\\
    \max_{1\leq j\leq n}\Big\|\frac{\partial P}{\partial y_{j}}\Big\|_{D(s,r)}  &\leq &\<P\>_{r,D(s,r)},
       \end{eqnarray*}
 \begin{eqnarray*}
\Big\|\frac{\partial P}{\partial w_{j}}\Big\|_{D(s,r)}  &\leq &\frac{r}{j^{\b}}\<P\>_{r,D(s,r)}, \quad \forall \,j\geq1\quad \text{and}\quad  w_{j}=z_{j}, \,\z_{j},\\
   \Big\|\frac{\partial^{2} P}{\partial w_{j}\partial w_{l}}\Big\|_{D(s,r)}  &\leq &\frac{1}{(jl)^{\b}}\<P\>_{r,D(s,r)},\quad \forall \,j,l\geq1\quad \text{and}\quad  w_{j}=z_{j}, \,\z_{j}.
   \end{eqnarray*}
$\bullet$ The semi-norm $\<\,\cdot\,\>^{\L}_{r,D(s,r)}$ is defined by the conditions 
    \begin{eqnarray*}
        \big \|P\big\|^{\L}_{D(s,r)}&\leq & r^{2}\<P\>^{\L}_{r,D(s,r)},\\
    \max_{1\leq j\leq n}\Big\|\frac{\partial P}{\partial y_{j}}\Big\|^{\L}_{D(s,r)}  &\leq &\<P\>^{\L}_{r,D(s,r)},
       \end{eqnarray*}
 \begin{eqnarray*}
\Big\|\frac{\partial P}{\partial w_{j}}\Big\|^{\L}_{D(s,r)}  &\leq &\frac{r}{j^{\b}}\<P\>^{\L}_{r,D(s,r)}, \quad \forall \,j\geq1\quad \text{and}\quad  w_{j}=z_{j}, \,\z_{j},\\
   \Big\|\frac{\partial^{2} P}{\partial w_{j}\partial w_{l}}\Big\|^{\L}_{D(s,r)}  &\leq &\frac{1}{(jl)^{\b}}\<P\>^{\L}_{r,D(s,r)},\quad \forall \,j,l\geq1\quad \text{and}\quad  w_{j}=z_{j}, \,\z_{j}.
   \end{eqnarray*}

 The last assumption is then the following
 \begin{assumption}[Decay]\label{AS4} 
    $P\in \G^{\b}_{r,D(s,r)}$ for some $\beta >0$.
     \end{assumption}

 \begin{rema}    
The control of the second derivative is the most important condition. The other ones are imposed so that we are able to recover the last one after the KAM iteration (see Lemma \ref{lem.2}).  Furthermore the assumptions on the first derivatives are already contained in Assumption \ref{AS3}  as soon as $p>0$.
\end{rema} 
~
\subsection{Statement of the abstract KAM Theorem}~\\ [5pt]
Recall that $M=|\om|^{\L}_{\Pi}+|\Omega|^{\L}_{2\b,\Pi}$.
 \begin{theo}\label{thmKAM}
 Suppose that $N$ is a family of Hamiltonians of the form \eqref{Ham.Comp} on the phase space $\mathcal{P}^{}$  depending on parameters $\xi \in \Pi$ so that Assumptions \ref{AS1} and \ref{AS2} are satisfied. Then there exist $\eps_0>0$ and $s>0$ so that every perturbation $H=N+P$  of $N$ which satisfies Assumptions \ref{AS3} and \ref{AS4} and the smallness condition
 \begin{equation*}
 \eps=\big({\bf \|}X_{P}{\bf\|}_{r,D(s,r)}+\<P\>_{r,D(s,r)}\big)+\frac{\alpha}{M}\big({\bf \|}X_{P}{\bf\|}^{\L}_{r,D(s,r)}+\<P\>^{\L}_{r,D(s,r)}\big)\leq \eps_0\alpha,
 \end{equation*}
 for some $r>0$ and $0<\alpha\leq 1$, the following holds. There exist
 \begin{enumerate}[(i)]
 \item a Cantor set $\Pi_\alpha\subset \Pi$ with $\text{Meas}(\Pi\backslash \Pi_\alpha)\rightarrow 0$ as $\alpha\rightarrow 0$ ;
\item a Lipschitz family of real analytic, symplectic coordinate transformations $\Phi: D(s/2,r/2)\times \Pi_\alpha\rightarrow D(s,r)$ ;
\item a Lipschitz family of new normal forms 
 \begin{equation*} 
 N^\star=\sum_{j=1}^n \omega_j^\star(\xi)  y_j+ \sum_{j\geq 1}\Omega_j^\star(\xi) z_j\bar z_j
 \end{equation*} 
 defined on $D(s/2,r/2)\times \Pi_\alpha $ ;
 \end{enumerate}
 such that 
 \begin{equation*}
 H\circ \Phi = N^\star +R^\star
 \end{equation*}
 where $R^\star$ is analytic on $D(s/2,r/2)$ and globally of order 3  at $\T^n\times\{0,0,0\}$. That is the Taylor expansion of $R^\star$ only contains monomials $y^mz^q\bar z^{\bar q}$ with $2|m|+|q+\bar q|\geq 3$.\\
 Moreover each symplectic coordinate transformation is close to the identity
 \begin{equation}\label{Est.Phi}
{\bf \|}\Phi-Id {\bf\|}_{r,D(s/2,r/2)}\leq c \eps , 
 \end{equation} 
 the new frequencies are close to the original ones
 \begin{equation}\label{Est.Freq}
|\om^{\star}-\om|_{\Pi_{\alpha}}+|\Omega^\star-\Omega|_{2\beta,\Pi_{\alpha}}\leq c\eps ,
\end{equation} 
 and the new frequencies satisfy a non resonance condition
  \begin{equation}\label{NewNR}
 \big|k\cdot \om^\star(\xi)+l\cdot\Omega^\star(\xi) \big|\geq \frac {\alpha} 2\  \frac{ \<l\>}{1+|k|^{\tau}},\quad (k,l)\in \mathcal{Z},\ \xi\in \Pi_\alpha .
 \end{equation}
 \end{theo}
As the consequence, for each $\xi\in\Pi_\alpha$ the torus $\Phi\big(\T^n\times\{0,0,0\}\big)$ is still invariant under the flow of the perturbed Hamiltonian $H=N+P$, the flow is linear ( in the new variables) on these tori  and furthermore all these  tori are linearly stable.

  \subsection{General strategy}\label{Strategy}~\\[5pt]
 The general strategy is the classical one used for instance in \cite{Kuk1,Kuk2,Poschel}. For convenience of the reader we recall it. Let $H=N+P$ be a Hamiltonian, where $N$ is given by \eqref{Ham.Comp} and $P$ a perturbation which satisfies the assumptions of the previous section. We then  consider the  second order Taylor approximation of $P$ which is
\begin{equation}\label{taylor}
R=\sum_{2|m|+|q+\ov{q}|\leq 2}\,\sum_{k\in\Z^{n}}R_{kmq\ov{q}}\,\e^{ik\cdot\theta}y^{m}z^{q}\z^{\ov{q}},
\end{equation} 
with $R_{kmq\ov{q}}=P_{kmq\ov{q}}$ and we define its mean value by 
\begin{equation*} 
[R]=\sum_{|m|+|q|=1}R_{0mqq} y^{m}z^{q}\z^{q}.
\end{equation*} 
Recall that in this setting $z,\ov{z}$ have homogeneity 1, whereas $y$ has homogeneity 2.\\
Let  $F$ be a function of the form \eqref{taylor} and denote by  $X^{t}_{F}$ the flow at time $t$ associated to the vector field of $F$. We can then define a new Hamiltonian by $H\circ X^{1}_{F}:=N_{+}+P_{+}$, and the Hamiltonian structure is preserved, because  $X^{1}_{F}$ is a symplectic transformation.
The idea of the KAM step is to find, iteratively, an adequate function $F$ so that the new error term has a small quadratic part. Namely, thanks to the Taylor formula we can write 
\begin{eqnarray*} 
H\circ X^{1}_{F}&=&N\circ X^{1}_{F}+(P-R)\circ X_{F}^{1}+R\circ X^{1}_{F} \\
&=&N+\cp{N,F}+        \int_{0}^{1}(1-t)\cp{\cp{N,F},F}\circ X^{t}_{F}\,\text{d}t+ \\
&&+  (P-R)\circ X_{F}^{1}+R+ \int_{0}^{1}\cp{R,F}\circ X^{t}_{F}\,\text{d}t.
  \end{eqnarray*}
In view of the previous equation, we define the new normal form  by $N_{+}=N+\wh{N}$, where $\wh{N}$ satisfies the so-called homological equation (the unknown are $F$ and $\wh{N}$)
 \begin{equation}\label{homol}
 \big\{F,N\big\}+\wh{N}=R.
  \end{equation}
The new normal form $N_+$ has the form \eqref{Ham.Comp} with new    frequencies   given by
$$\om^{+}(\xi)=\om(\xi)+\wh{\om}(\xi) \mbox{ and }\Omega^{+}(\xi)=\Omega(\xi)+\wh{\Omega}(\xi)$$ where  
\begin{equation}\label{newfreq} \dis \wh{\om}_{j}(\xi)=\frac{\partial \wh{N}}{\partial y_{j}}(0,0,0,0,\xi) ) \mbox{ and } \dis \wh{\Omega}_{j}(\xi)=\frac{\partial^{2} \wh{N}}{\partial z_{j} \partial \ov{z}_{j}}(0,0,0,0,\xi). \end{equation}
 Once the homological equation is solved, we define the  new perturbation term $P_{+}$ by 
\begin{equation} \label{newP}
P_{+}=(P-R)\circ X_{F}^{1}+\int_{0}^{1}\cp{R(t),F}\circ X^{t}_{F}\,\text{d}t,
  \end{equation}
 where $R(t)=(1-t)\wh{N}+tR$ in such a way that 
 $$H\circ X^{1}_{F}=N_+ +P_+\  .$$
 Notice that if $P$ was initially  of size $\eps$, then $R$ and $F$ are of size $\eps$, and the quadratic part of $P_{+}$ is formally of size $\eps^{2}$. That is, the formal iterative scheme is exponentially convergent.\\
 
 Without any smoothing effect on the regularity, there is no decreasing property in the correction term added to the external frequencies \eqref{newfreq}. In that case it would be impossible to control the small divisors (see \eqref{dio}) at the next step. In this work the smoothing condition \eqref{XP} on $X_{P}$ is replaced by Assumption \ref{AS4} (see also Remark \ref{Rq4}). The difficulty is to verify the conservation of this assumption at each step.

 \subsection*{ Plan of the proof of Theorem \ref{thmKAM}} In Section \ref{Sect.Lin} we solve the homological equation and give estimates on the solutions. Then we study  precisely the flow map $X_{F}^{t}$ and the composition $H\circ X_{F}^{1}$. In Section \ref{Sect.KAM} we  estimate the new error term and the new frequencies after the KAM step, and Section \ref{Sect.Conv} is devoted to the convergence of the KAM method and the proof of Theorem \ref{thmKAM}. 
 
  \begin{enonce*}{Notations}
 In this paper $c$, $C$ denote constants the value of which may change
from line to line. These constants will always be universal, or depend on the fixed quantities $n,\beta,\Pi,p
$.\\
We denote by $\N$ the set of the non negative integers, and $\N^{*}=\N\backslash\{0\}$.
For $\l =(l_{1},\dots, l_{k},\dots)\in \Z^{\infty}$, we denote by $\dis |l|=\sum_{j=1}^{\infty}|l_{j}|$ its length (if it is finite), and $\dis \<l\>=1+|\sum_{j=1}^{\infty}jl_{j}|$. We define the space $\mathcal{Z}=\big\{ (k,l)\neq 0, \,k\in \Z^{n}, l\in \Z^{\infty},\,|l|\leq 2\,\big\}$. The notation $\text{Meas}$ stands for the  Lebesgue measure in $\R^{n}$.
 \end{enonce*}
 ~

In the sequel,  we will state without proof some intermediate results of \cite{Poschel} which still hold under our conditions ; hence the reader should refer to \cite{Poschel} for the details.  For the convenience of the reader we decided to remain as close as possible to the notations of J. P\"oschel.


\section{The linear step}\label{Sect.Lin}~
In this section, we solve equation \eqref{homol} and study the Lie transform $X^{t}_{F}$.\\
Following \cite{Poschel}, $\|\cdot\|^{*}$ (respectively $\<\,\cdot\,\>^{*}$\,) stands either for $\|\cdot\|$ or $\|\cdot\|^{\L}$ (respectively $\<\,\cdot\,\>$ or $\<\,\cdot\,\>^{\L}$\,) and $\|\cdot\|^{\lambda}$ stands for $\|\cdot\|+\lambda \|\cdot\|^{\L}$.\\

\subsection{The homological equation}~\\[5pt]
  The following result shows that it is possible to solve equation \eqref{homol} under the  Diophantine condition \eqref{dio}.
\begin{lemm}[\cite{Poschel}]\label{lem.Poschel1}
Assume that the frequencies satisfy, uniformly on $\widetilde{\Pi}_{\a}$, for some $\a>0$ the condition \eqref{dio}.
 Then the homological  equation \eqref{homol} has a solution $F$, $\wh{N}$ which is normalised by $[F]=0$, $[\wh{N}]=\wh{N}$, and satisfies   for all $0<\s<s$,  and $0\leq \lambda\leq \alpha/M$
  \begin{equation*}
{\bf \|} X_{\wh{N}}{\bf\|}^{*}_{r,D(s,r)}\leq {\bf \|}X_{R}{\bf\|}^{*}_{r,D(s,r)},\quad {\bf \|} X_{F}{\bf\|}^{\lambda}_{r,D(s-\s,r)}\leq \frac{C}{\alpha \s^{t}}{\bf \|}X_{R}{\bf\|}^{\lambda}_{r,D(s,r)},
 \end{equation*}
 where $t$ only depends on $n$ and $\tau$.
  \end{lemm}
The space $\G^{\b}_{r,D(s,r)}$ is not stable under the Poisson bracket. Therefore we need to introduce the  space  $\G^{\b,+}_{r,D(s,r)} \subset \G^{\b}_{r,D(s,r)}$ endowed with the norm $\<\,\cdot\, \>^{+}_{r,D(s,r)}+\<\,\cdot\, \>^{+,\L}_{r,D(s,r)}$ defined by the following conditions.   
\begin{equation*}
 \big\|F\big\|^{*}_{D(s,r)}\leq  r^{2}\<F\>^{+,*}_{r,D(s,r)},\quad \max_{1\leq j\leq n}\Big\|\frac{\partial F}{\partial y_{j}}\Big\|^{*}_{D(s,r)}  \leq \<F\>^{+,*}_{r,D(s,r)},
\end{equation*}
  \begin{eqnarray*}
\Big\|\frac{\partial F}{\partial w_{j}}\Big\|^{*}_{D(s,r)}  &\leq &\frac{r}{j^{\b+1}}\<F\>^{+,*}_{r,D(s,r)}, \quad \forall \,j\geq1\quad \text{and}\quad  w_{j}=z_{j}, \,\z_{j},\\
    \Big\|\frac{\partial^{2} F}{\partial w_{j}\partial w_{l}}\Big\|^{*}_{D(s,r)}  &\leq &\frac{1}{(jl)^{\b}(1+|j-l|)}\<F\>^{+,*}_{r,D(s,r)}\quad \forall \,j,l\geq1\quad \text{and}\quad  w_{j}=z_{j}, \,\z_{j}.
   \end{eqnarray*}
   ~\\
This definition is motivated by the following result, which can be understood as a smoothing property of the homological equation

\begin{lemm}\label{lem.1}
Assume that the frequencies satisfy \eqref{dio}, uniformly on $\widetilde{\Pi}_{\a}$. Let $F, \wh{N}$ be given by Lemma \ref{lem.Poschel1}. Assume moreover that $R\in \G^{\b}_{r,D(s,r)}$, then  there exists $C>0$ so that for any $0<\s<s$, we have $F\in \G^{\b,+}_{r,D(s-\s,r)}$,  $\wh{N}\in \G^{\b}_{r,D(s-\s,r)}$  and 
 \begin{equation*}
 \<F\>^{+}_{r,D(s-\s,r)}\leq \frac{C}{\a\s^{t}} \<R\>_{r,D(s,r)},
   \end{equation*}
   \begin{equation}\label{Delta+}
 \<F\>^{+,\L}_{r,D(s-\s,r)}\leq \frac{C}{\a \s^{t}} \Big( \<R\>_{r,D(s,r)} + \frac{M}{\alpha}\<R\>^{\L}_{r,D(s,r)}\Big),
   \end{equation}
  and 
   \begin{equation*}
 \<\wh{N}\>_{r,D(s-\s,r)}\leq  \<R\>_{r,D(s,r)},\quad  \<\wh{N}\>^{\L}_{r,D(s-\s,r)}\leq  \<R\>^{\L}_{r,D(s,r)},
  \end{equation*}
  where $t$ only depends on $n$ and $\tau$.
  \end{lemm}

For the proof of this result, we need  the classical lemma
\begin{lemm}\label{lem.1.0}
Let $f\,:\,\R\longrightarrow \C$ be a periodic function and assume that $f$ is holomorphic in the domain $|\text{Im}\, \theta|<s$, and continuous on $|\text{Im}\, \theta|\leq s$. Then there exists $C>0$ so that its Fourier coefficients satisfy 
\begin{equation*}
|\wh{f}(k)|\leq C\e^{-|k|s}\sup_{{|\text{Im}\, \theta|<s}}|f(\theta)|.
\end{equation*}
 \end{lemm}
  \begin{proof}[Proof of Lemma \ref{lem.1}] In \cite{Poschel}, the author looks for a solution $F$ of \eqref{homol} of the   form of \eqref{taylor}, i.e.
\begin{equation}\label{def.F}
F=\sum_{2|m|+|q+\ov{q}|\leq 2}\,\sum_{k\in\Z^{n}}F_{kmq\ov{q}}\,\e^{ik\cdot\theta}y^{m}z^{q}\z^{\ov{q}}.
\end{equation}
A direct computation then shows that the coefficients in \eqref{def.F} are given by 
\begin{equation}\label{solution} 
iF_{kmq\ov{q}}= \left\{
\begin{array}{ll} 
\dis \frac{R_{kmq\ov{q}}}{k\cdot \om+(q-\ov{q})\cdot \Omega},\quad &\text{if} \quad |k|+|q-\ov{q}|\neq 0, \\[9pt]  
0 , \;\,  &\text{otherwise},
\end{array} 
\right.
\end{equation}
and that we can set $\wh{N}=[R]$.\\[3pt]
In the following we will use the notation $q_{j}=(0,\cdots,0,1,0,\cdots)$, where the 1 is at the $j^{th}$ position, and   $q_{jl}=q_{j}+q_{l}$.\\[3pt]
The variables $z$ and $\ov{z}$ exactly play the same role, therefore it is enough to study the derivatives in the variable $z$.\\[5pt]
In the sequel we write $A_{k}=1+|k|^{\tau}$. Then it easy to check that for any $j\geq 1$ and $\s>0$, 
$$\sum_{k\in \Z^{n}}A_{k}^{j}\e^{-|k|\s}\leq \frac{C}{\s^{t}},$$
for some $C>0$ and $t=2j\tau+n+1$. In the sequel, $t$ may vary from line to line, but will remain independent of $\s$.\\[5pt]
$\spadesuit$ We first prove that 
$\dis \<F\>^{+}_{r,D(s-\s,r)}\leq \frac{C }{\a \s^{t}} \<R\>_{r,D(s,r)}.$\\
 $\bullet$ Observe that
$\dis \frac{\partial^{2}R}{\partial z_{j}\partial z_{l}}=\sum_{k\in\Z^{n}}R_{k\,0\,q_{jl}\,0}\,\e^{ik\cdot\theta}$,
then according to Lemma \ref{lem.1.0},  there exists $C>0$ so that 
$\dis |R_{k\,0\,q_{jl}\,0}|\leq C\frac{\<R\>_{r,D(s,r)}\e^{-|k|s}}{(jl)^{\b}}$,
and thus by \eqref{solution} and \eqref{dio}
\begin{equation} \label{Fkq}
|F_{k\,0\,q_{jl}\,0}|\leq C\frac{A_{k}}{\alpha}\frac{\<R\>_{r,D(s,r)}\e^{-|k|s}}{(jl)^{\b}(1+|j-l|)}.
\end{equation}
Therefore, as we also have 
\begin{equation} \label{D2F}
\frac{\partial^{2}F}{\partial z_{j}\partial z_{l}}=\sum_{k\in\Z^{n}}F_{k\,0\,q_{jl}\,0}\,\e^{ik\cdot\theta},
\end{equation}
we deduce that 
\begin{eqnarray} 
\Big \|\frac{\partial^{2}F}{\partial z_{j}\partial z_{l}}\Big \|_{D(s-\s,r)} &\leq &  \sum_{k\in\Z^{n}}|F_{k\,0\,q_{jl}0}|\e^{|k|(s-\s)}\nonumber\\
 &\leq & \frac{C\<R\>_{r,D(s,r)}}{\alpha (jl)^{\b}(1+|j-l|)}\sum_{k\in\Z^{n}}A_{k}\e^{-|k|\s}\nonumber \\
 &\leq &  \frac{C\<R\>_{r,D(s,r)}}{\alpha \s^{t} (jl)^{\b}(1+|j-l|)}. \label{F0}
  \end{eqnarray}
   $\bullet$ We compute
 \begin{equation} \label{decomp}
\frac{\partial F}{\partial z_{j}}=\sum_{k\in\Z^{n}}F_{k\,0\,q_{j}\,0}\e^{ik\cdot\theta}+\sum_{k\in\Z^{n},\, l\geq 1}F_{k\,0\,q_{j}\ov{q}_{l}}\e^{ik\cdot\theta}\z_{l}+2\sum_{k\in\Z^{n}}F_{k\,0\,2q_{j}\,0}\e^{ik\cdot\theta}z_{j}.
\end{equation}
 Now  observe that
$\dis \big(\frac{\partial R}{\partial z_{j}}\big)_{|z=\z=0}=\sum_{k\in\Z^{n}}R_{k\,0\,q_{j}\,0}\,\e^{ik\cdot\theta},$
 then by Lemma \ref{lem.1.0},
 \begin{eqnarray*} 
|R_{k\,0\,q_{j}\,0}|&\leq& C \e^{-|k|s}\, \sup_{|\text{Im}\,\theta|<s}\Big|\big(\frac{\partial R}{\partial z_{j}}\big)_{|z=\z=0}\Big|\\
&\leq &C \,\e^{-|k|s} \Big \|\frac{\partial R}{\partial z_{j}}\Big \|_{D(s,r)} \leq Cr\frac{ \,\e^{-|k|s}}{j^{\b}} \<R\>_{r,D(s,r)}.
\end{eqnarray*}
  From the previous estimate, \eqref{solution} and \eqref{dio} we get 
  \begin{equation*}
  |F_{k\,0\,q_{j}\,0}|\leq \frac{A_{k}}{\alpha(1+j)}|R_{k\,0\,q_{j}\,0}|\leq \frac{CrA_{k} \,\e^{-|k|s}}{\alpha j^{\b}(1+j)} \<R\>_{r,D(s,r)}
  \end{equation*}
 and thus
\begin{eqnarray} 
\Big  \|\sum_{k\in\Z^{n}}F_{k\,0\,q_{j}0}\e^{ik\cdot\theta}\Big \|_{D(s-\s,r)} &\leq &  \sum_{k\in\Z^{n}}|F_{k\,0\,q_{j}0}|\e^{|k|(s-\s)}\nonumber \\
 &\leq & Cr\frac{\<R\>_{r,D(s,r)}}{\alpha j^{\b}(1+j)}\sum_{k\in\Z^{n}}A_{k}\e^{-|k|\s}\nonumber \\
 &\leq & \frac{C r \<R\>_{r,D(s,r)}}{\alpha  \s^{t}j^{\b}(1+j)}. \label{F1}
 \end{eqnarray}
 Similarly, we have 
$ \dis  |F_{k\,0\,2q_{j}0}|  \leq \frac{CrA_{k} \,\e^{-|k|s}}{\alpha j^{\b}(1+j)} \<R\>_{r,D(s,r)},$
  which leads to 
   \begin{equation} \label{F2}
  \Big \|\sum_{k\in\Z^{n}}F_{k\,0\,2q_{j}0}\e^{ik\cdot\theta}\Big \|_{D(s-\s,r)}  \leq  \frac{Cr\<R\>_{r,D(s,r)}}{\alpha \s^{t} j^{\b}(1+j)}.   \end{equation}
By Cauchy-Schwarz in the variable $l$ and \eqref{D2F}, \eqref{F0}
\begin{eqnarray} 
\Big  \|\sum_{k\in\Z^{n},\, l\geq 1}F_{k\,0\,q_{j}\ov{q}_{l}}\e^{ik\cdot\theta}\z_{l}\Big \|_{D(s-\s,r)} &\leq & \Big( \sum_{l\geq 1}\Psi^{-2}(l)|\sum_{k\in\Z^{n}}F_{k\,0\,q_{j}\ov{q}_{l}}\e^{ik\cdot\theta}|^{2}\Big)^{\frac12}\Big(\sum_{l\geq 1}|z_{l}|^{2}\Psi^{2}(l)\Big)^{\frac12} \nonumber \\
 &\leq &   \frac{Cr }{\alpha \s^{t} j^{\b} } \Big( \sum_{l\geq 1}\frac{1}{ l^{2\b} \Psi^{2}(l)(1+|j-l|)^{2}}\Big)^{\frac12} \<R\>_{r,D(s,r)}  \nonumber \\
 &\leq & \frac{Cr\<R\>_{r,D(s,r)}}{\alpha \s^{t} j^{\b}(1+j)}, \label{F3}
 \end{eqnarray}
 since $\Psi(l)\geq l$.\\
  Finally, inserting \eqref{F1}, \eqref{F2} and \eqref{F3} in \eqref{decomp} we obtain 
   \begin{equation}\label{dz} 
  \Big \|\frac{\partial F}{\partial z_{j}}\Big \|_{D(s-\s,r)}  \leq  \frac{Cr\<R\>_{r,D(s,r)}}{\alpha \s^{t} j^{\b}(1+j)}.
   \end{equation} 
  $\bullet$  We can  write 
$\dis \frac{\partial F}{\partial y_{j}}=\sum_{k\in\Z^{n}}F_{km_{j}0\,0}\e^{ik\cdot\theta}$.
 Hence by \eqref{solution} and \eqref{dio},
$\dis |F_{km_{j}0\,0}|\leq \frac{A_{k}}{\alpha}|R_{km_{j}0\,0}|$, and 
thanks to Lemma \ref{lem.1.0} applied to the series $\dis \frac{\partial R}{\partial y_{j}}=\sum_{k\in\Z^{n}}R_{km_{j}0\,0}\e^{ik\cdot\theta}$,
 \begin{equation}\label{Fkm}  
|F_{km_{j}0\,0}|\leq C \frac{A_{k}}{\alpha}\e^{-|k|s}\<R\>_{r,D(s,r)},
\end{equation} 
 and  we obtain 
 \begin{equation}\label{dy}
\Big  \|\frac{\partial F}{\partial y_{j}}\Big \|_{D(s-\s,r)} \leq  \sum_{k\in\Z^{n}}|F_{km_{j}0\,0}|\e^{|k|(s-\s)} \leq \frac{C }{\alpha \s^{t}}    \<R\>_{r,D(s,r)}. 
 \end{equation}
$\bullet$  To obtain the bound for $ \|F\|_{D(s-\s,r)}$ write
  \begin{multline} \label{decomp3}
F=\sum_{k\in\Z^{n}}F_{k\,0\,0\,0}\e^{ik\cdot\theta}+\sum_{k\in\Z^{n},1\leq j\leq n}F_{k\,m_{j}\,0\,0}\e^{ik\cdot\theta}y_{j}+\\
\sum_{k\in\Z^{n},j,l\geq 1}F_{k\,0\,q_{jl}\,0}\e^{ik\cdot\theta}z_{j}z_{l}+\sum_{k\in\Z^{n},j,l\geq 1}F_{k\,00\,q_{jl}}\e^{ik\cdot\theta}\ov{z}_{j}\ov{z}_{l}+\sum_{k\in\Z^{n},j,l\geq 1}F_{k\,0q_{j}q_{l}}\e^{ik\cdot\theta}\ov{z}_{j}{z}_{l}.
\end{multline}
Since $ \dis R_{|y=z=\ov{z}=0}=\sum_{k\in\Z^{n}}R_{k\,0\,0\,0}\e^{ik\cdot\theta}$, by Lemmas \ref{lem.1.0} and \ref{lem.Poschel1} we deduce that 
  \begin{equation}\label{Fk0}  
\big|F_{k0\,0\,0}\big |\leq Cr^{2}\frac{A_{k}}{\alpha}\e^{-|k|s}\<R\>_{r,D(s,r)},
\end{equation}
 hence, thanks to \eqref{Fkm} and \eqref{Fk0} we can bound the sums of the first line in \eqref{decomp3} as in the previous point.\\
 Now thanks to \eqref{Fkq} and to the Cauchy-Schwarz inequality we have
  \begin{eqnarray*} 
 \Big \|\sum_{k\in\Z^{n},\, j,l\geq 1}F_{k\,0\,q_{jl}0}\e^{ik\cdot\theta}z_{j}z_{l}\Big \|_{D(s-\s,r)} &\leq &\frac{C \<R\>_{r,D(s,r)}}{\alpha \s^{t}} \sum_{j,l\geq 1}\frac{|z_{j}z_{l}|}{(jl)^{\b}(1+|j-l|)}  \nonumber \\
 &\leq & \frac{C \<R\>_{r,D(s,r)}}{\alpha \s^{t}}\Big( \sum_{j\geq 1}\frac{|z_{j}|}{j^{\b}} \Big)^{2} \nonumber \\
 &\leq & \frac{C \<R\>_{r,D(s,r)}}{\alpha \s^{t}}\Big( \sum_{j\geq 1}\Psi^{2}(j)|z_{j}|^{2} \Big)\Big( \sum_{j\geq 1}\frac{1}{j^{2\b}\Psi^{2}(j) } \Big)  \nonumber \\
 &\leq  &\frac{Cr^{2} \<R\>_{r,D(s,r)}}{\alpha \s^{t}}.
 \end{eqnarray*}
Therefore we proved that  $ \dis \|F\|_{D(s-\s,r)} \leq  \frac{Cr^{2} \<R\>_{r,D(s,r)}}{\alpha \s^{t}}$. \\
This latter estimate  together with the estimates \eqref{F0}, \eqref{dz} and  \eqref{dy} shows that 
 \begin{equation*}
 \<F\>^{+}_{r,D(s-\s,r)}\leq \frac{C }{\a \s^{t}} \<R\>_{r,D(s,r)}.
  \end{equation*}
  $\spadesuit$ We now show that 
  \begin{equation}\label{est.N}
 \<\wh{N}\>_{r,D(s-\s,r)}\leq  \<R\>_{r,D(s,r)}.
  \end{equation}
Since $\wh{N}=[R]$ we have 
\begin{equation}\label{N1}
\wh{N}=\sum_{j=1}^{n}R_{0m_{j}00}\,y_{j}+\sum_{j\geq 1}R_{00q_{j}q_{j}}\,z_{j}\z_{j},
\end{equation}
and we can observe that 
\begin{equation}\label{N2}
R_{0m_{j}00}=\frac{1}{(2\pi)^{n}}\int_{\theta\in \T^{n}}\frac{\partial R}{\partial y_{j}}(\theta,0,0,0)\text{d}\theta, \quad R_{00q_{j}q_{j}}=\frac{1}{(2\pi)^{n}}\int_{\theta\in \T^{n}}\frac{\partial^{2} R}{\partial z_{j}\partial \z_{j}}(\theta,0,0,0)\text{d}\theta,
\end{equation}
which imply the bounds $\dis |R_{0m_{j}00}|\leq \<R\>_{r, D(s,r)}$ and $\dis |R_{00q_{j}q_{j}}|\leq \<R\>_{r, D(s,r)}/j^{2\beta}$ and thus \eqref{est.N}.\\[5pt]
$\spadesuit$ It remains to check the estimates with the Lipschitz semi norms.\\[3pt]
As in \cite{Poschel}, for $|k|+|q_{j}-\ov{q_{l}}|\neq 0$ define $\delta_{k,jl}=k\cdot \om+\Omega_{j}-\Omega_{l}$. Then by \eqref{solution},
 $$i\Delta_{\xi\eta}F_{kmq_{j}\ov{q}_{l}}=\delta^{-1}_{k,jl}(\eta)\Delta_{\xi\eta}R_{kmq_{j}\ov{q}_{l}}+R_{kmq_{j}\ov{q}_{l}}(\xi) \Delta_{\xi\eta} \delta^{-1}_{k,jl}.$$
By \eqref{dio}, $| \delta^{-1}_{k,jl}|\leq A_{k}/\alpha$ and thus 
\begin{equation*}
| \Delta_{\xi\eta} \delta^{-1}_{k,jl}|\leq \frac{A^{2}_{k}}{\alpha^{2}}\big(|k||\Delta_{\xi\eta} \om|+|\Delta_{\xi\eta}\Omega_{j}|+|\Delta_{\xi\eta}\Omega_{l}|\big),
\end{equation*}
hence 
\begin{equation*}
\frac{| \Delta_{\xi\eta} \delta^{-1}_{k,jl}|}{|\xi-\eta|}\leq C\frac{kA^{2}_{k}}{\alpha^{2}}\big(|\omega|_{\Pi}^{\L}+|\Omega|^{\L}_{2\beta,\Pi}\big)\leq CM\frac{kA^{2}_{k}}{\alpha^{2}},
\end{equation*}
and we have 
\begin{equation}\label{Delta}
\frac{| \Delta_{\xi\eta} F_{kmq_{j}\ov{q}_{l}}|}{|\xi-\eta|}\leq  C \frac{kA_{k}}{\alpha}\Big(\frac{| \Delta_{\xi\eta} R_{kmq_{j}\ov{q}_{l}}|}{|\xi-\eta|}     + \frac{M}{\alpha}\big|R_{kmq_{j}\ov{q}_{l}}(\xi)\big|\Big).
\end{equation}
 Thanks to the estimate \eqref{Delta} it is easy to obtain \eqref{Delta+}.\\[3pt]
Finally, the estimate $ \<\wh{N}\>^{\L}_{r,D(s-\s,r)}\leq  \<R\>^{\L}_{r,D(s,r)}$ is a straightforward consequence of \eqref{N1} and \eqref{N2}.
\end{proof}~

\subsection{Estimates on the Poisson bracket}~\\[5pt]
\begin{lemm}\label{lem.2}
Let $R\in \G^{\b}_{r,D(s,r)}$ and  $F\in \G^{\b,+}_{r,D(s,r)}$ be both of degree 2, i.e. of the form \eqref{taylor}. Then there exists $C>0$ so that for any $0<\s<s$ 
\begin{equation}\label{Est.crochet}
\<\,\big\{R,F\big\}\,\>_{r,D(s-\s,r)} \leq \frac{C}{\s}\<R\>_{r,D(s,r)}\<F\>^{+}_{r,D(s,r)},
  \end{equation}
  and 
  \begin{equation*}
\<\,\big\{R,F\big\}\,\>^{\L}_{\b,{D}(s-\s,r)} \leq \frac{C}{\s}\Big(\<R\>_{r,D(s,r)}\<F\>^{+,\L}_{r,D(s,r)}+\<F\>^{+}_{r,D(s,r)}\<R\>^{\L}_{r,D(s,r)}\Big).
  \end{equation*}
  \end{lemm}

\begin{proof} The expansion of  $\cp{R,F}$ reads
 \begin{equation*} 
 \cp{R,F}=\sum_{k=1}^{n}\Big(\frac{\partial R}{\partial \theta_{k}} \frac{\partial F}{\partial y_{k}}-\frac{\partial R}{\partial y_{k}} \frac{\partial F}{\partial \theta_{k}}   \Big)+i\sum_{k\geq 1}\Big(\frac{\partial R}{\partial z_{k}} \frac{\partial F}{\partial \z_{k}}-\frac{\partial R}{\partial \z_{k}} \frac{\partial F}{\partial z_{k}}   \Big).
  \end{equation*}
It remains to estimate each term of this expansion and its derivatives. We will control the derivative with respect to $\theta_{k}$ thanks to the 
 Cauchy formula :
\begin{equation}\label{Cauchy}
\Big \|  \frac{\partial P}{\partial \theta_{k}}\Big \|_{D(s-\s,r)} \leq \frac{C}{\s} \big \|  P\big \|_{D(s,r)} ,
 \end{equation}
 which explains the loss of $\s$.\\
 Notice that if $P$ is of degree 2 (and that is the case for $F$ and $R$) we have 
 \begin{equation}\label{vanish}
 \frac{\partial^{2} P}{\partial z\partial y}=\frac{\partial^{2} P}{\partial y^{2}}=\frac{\partial^{3} P}{\partial z^{3}}=0,
 \end{equation}
 fact which will be crucially used in the sequel. Finally observe that $z$ and $\ov{z}$ exactly play the same role, hence we will only take $\dis \frac{\partial}{\partial z}$ into consideration. \\[5pt]
 $\spadesuit$ We first prove \eqref{Est.crochet}.\\[4pt]
$\bullet$ Since $\dis \|  P\,Q\|_{D(s,r)}\leq \|  P\|_{D(s,r)} \|  Q\|_{D(s,r)}$  we have by Cauchy formula
\begin{eqnarray}
\big \| \cp{R,F}\big \|_{D(s-\s,r)}
&\leq& \frac{Cr^{2}}{\s} (2n+\sum_{k\geq 1}\frac{1}{k^{2\b+1}} )\<R\>_{r,D(s,r)}\<F\>^{+}_{r,D(s,r)}\nonumber \\
&\leq& \frac{Cr^{2}}{\s} \<R\>_{r,D(s,r)}\<F\>^{+}_{r,D(s,r)}.\label{dx0}
\end{eqnarray}
$\bullet$ With \eqref{Cauchy} we have 
    \begin{eqnarray*}
\Big\|  \frac{\partial}{\partial y_{j}} \Big(\frac{\partial R}{\partial \theta_{k}} \frac{\partial F}{\partial y_{k}}\Big)\Big\|_{D(s-\s,r)}&\leq &\Big\|  \frac{\partial}{\partial \theta_{k}} \Big(\frac{\partial R}{\partial y_{j}}\Big)  \Big\|_{{D}(s-\s,r)} \Big\|  \frac{\partial F}{\partial y_{k}}\Big\|_{D(s,r)}\\
&\leq& \frac{C}{\s} \Big\| \frac{\partial R}{\partial y_{j}}  \Big\|_{D(s,r)} \Big\|  \frac{\partial F}{\partial y_{k}}\Big\|_{D(s,r)}\\
&\leq &\frac{C}{\s}\<R\>_{r,D(s,r)}\<F\>^{+}_{r,D(s,r)},
 \end{eqnarray*}
and the same estimate holds interchanging $R$ and $F$. In view of \eqref{vanish} we deduce 
 \begin{equation}\label{dx1}
\max_{1\leq y\leq n}\Big\|\frac{\partial}{\partial y_{j}} \cp{R,F}\Big\|_{D(s,r)}\leq  \frac{C}{\s} \<R\>_{r,D(s,r)}\<F\>^{+}_{r,D(s,r)}.
\end{equation}

\noindent $\bullet$ By \eqref{vanish}, 
$ \dis \frac{\partial}{\partial z_{j}} \Big(\frac{\partial R}{\partial y_{k}} \frac{\partial F}{\partial \theta_{k}}\Big)= \frac{\partial R}{\partial y_{k}}  \frac{\partial ^{2}F}{ \partial z_{j}\partial\theta_{k}}$, and by \eqref{Cauchy}
\begin{eqnarray*}
\Big\| \frac{\partial R}{\partial y_{k}}  \frac{\partial ^{2}F}{ \partial z_{j}\partial\theta_{k}} \Big\|_{D(s-\s,r)}
&\leq & \frac{C}{\s} \Big\|  \frac{\partial R}{\partial y_{k}}\Big\|_{D(s,r)}\Big\| \frac{\partial F}{\partial z_{j}}  \Big\|_{D(s,r)}\\
&\leq &\frac{Cr}{j^{\b}\s}\<R\>_{r,D(s,r)}\<F\>^{+}_{r,D(s,r)}.
 \end{eqnarray*}
Similarly
$\dis \Big\|  \frac{\partial}{\partial z_{j}} \Big(\frac{\partial R}{\partial \theta_{k}} \frac{\partial F}{\partial y_{k}}\Big)\Big\|_{D(s-\s,r)}\leq \frac{Cr}{j^{\b}\s}\<R\>_{r,D(s,r)}\<F\>^{+}_{r,D(s,r)}.$ 
By the Leibniz rule
\begin{multline*}
\Big\|  \frac{\partial}{\partial z_{j}} \Big(\frac{\partial R}{\partial z_{k}} \frac{\partial F}{\partial \z_{k}}\Big)\Big\|_{D(s,r)}\leq \\
\begin{aligned}
&\leq \Big\| \frac{\partial^{2} R}{\partial z_{k}\partial z_{j}}   \Big\|_{{D}(s,r)} \Big\|  \frac{\partial F}{\partial \z_{k}}\Big\|_{D(s,r)}+\Big\|   \frac{\partial^{2} F}{\partial z_{j} \partial \z_{k}} \Big\|_{{D}(s,r)} \Big\|  \frac{\partial R}{\partial z_{k}}\Big\|_{D(s,r)}\\
&\leq \frac{Cr}{j^{\b}}\Big( \frac{1}{k^{2\b+1}}+ \frac{1}{k^{2\b}(1+|j-k|)}\Big)\<R\>_{r,D(s,r)}\<F\>^{+}_{r,D(s,r)},
\end{aligned}
 \end{multline*}
and taking   the sum in $k$ yields
 \begin{equation*}
 \sum_{k\geq 1}   \Big\|  \frac{\partial}{\partial z_{j}} \Big(\frac{\partial R}{\partial z_{k}} \frac{\partial F}{\partial \z_{k}}\Big)\Big\|_{D(s,r)}\leq \frac{Cr}{j^{\b} \s} \<R\>_{r,D(s,r)}\<F\>^{+}_{r,D(s,r)}.
 \end{equation*}
 The previous estimates imply that 
  \begin{equation}\label{dx2}
\Big\|\frac{\partial}{\partial z_{j}} \cp{R,F}\Big\|_{D(s-\s,r)}\leq  \frac{Cr}{j^{\b} \s}\<R\>_{r,D(s,r)}\<F\>^{+}_{r,D(s,r)}.
\end{equation}
$\bullet$  Thanks to \eqref{vanish}, 
$  \dis \frac{\partial^{2}}{\partial z_{j}\partial z_{l}}  \Big(\frac{\partial R}{\partial y_{k}} \frac{\partial F}{\partial \theta_{k}}\Big)=
    \frac{\partial R}{\partial y_{k}}\frac{\partial^{3}F}{\partial z_{j}\partial z_{l}\partial \theta_{k}}$,
 and by \eqref{Cauchy} we obtain 
 \begin{eqnarray}
\Big\| \frac{\partial^{2}}{\partial z_{j}\partial z_{l}}  \Big(\frac{\partial R}{\partial y_{k}} \frac{\partial F}{\partial \theta_{k}}\Big)
\Big\|_{D(s-\s,r) }&\leq& \Big\| \frac{\partial R}{\partial y_{k}} 
\Big\|_{{D}(s,r) }\Big\| \frac{\partial^{3}F}{\partial z_{j}\partial z_{l}\partial \theta_{k}}\Big\|_{D(s-\s,r) }\nonumber\\
&\leq & \frac{C}{(jl)^{\b}\s}\<R\>_{r,D(s,r)}\<F\>^{+}_{r,D(s,r)}, \label{33}
\end{eqnarray}
and the same estimate holds interchanging $R$ and $F$.\\[3pt]
On the other hand, 
\begin{equation*}
   \frac{\partial^{2}}{\partial z_{j}\partial z_{l}}  \Big(\frac{\partial R}{\partial z_{k}} \frac{\partial F}{\partial \z_{k}}\Big)= 
\frac{\partial ^{2}R}{\partial z_{j}\partial z_{k}}  \frac{\partial^{2} F}{\partial z_{l}\partial \z_{k}}+\frac{\partial ^{2}R}{\partial z_{l}\partial z_{k}}  \frac{\partial^{2} F}{\partial z_{j}\partial \z_{k}},
 \end{equation*}
and
 \begin{eqnarray*}
\Big\|  \frac{\partial^{2}R}{\partial z_{j}\partial z_{k}}  \frac{\partial^{2} F}{\partial z_{l}\partial \z_{k}}\Big\|_{D(s-\s,r)}&\leq &\Big\|  \frac{\partial^{2} R}{\partial z_{j}\partial z_{k}} \Big\|_{D(s,r)} \Big\|  \frac{\partial^{2} F}{\partial z_{l}\partial \z_{k}}\Big\|_{D(s,r)}\\ 
&\leq &\frac{C}{(jlk^{2})^{\b}(1+|l-k|)}\<R\>_{r,D(s,r)}\<F\>^{+}_{r,D(s,r)}.
\end{eqnarray*}
Hence, with \eqref{33} we conclude that 
 \begin{equation}\label{dx3}
\Big\|\frac{\partial^{2}}{\partial z_{j}\partial z_{l}} \cp{R,F}\Big\|_{D(s-\s,r)}\leq  \frac{C}{(jl)^{\b} \s}\<R\>_{r,D(s,r)}\<F\>^{+}_{r,D(s,r)},
\end{equation}
 as the series $\dis \sum_{k\geq 1 } \frac1{k^{2\b}(1+|l-k|)}$ converges.\\[5pt]
Finally, the estimates \eqref{dx0},  \eqref{dx1},  \eqref{dx2} and  \eqref{dx3} yield the estimate \eqref{Est.crochet}.\\[5pt]
$\spadesuit$ To prove the estimate with the Lipschitz norms, we can use the previous analysis and the two following facts. \\[2pt]
Firstly, since $\Delta_{\xi\eta}(fg)=f(\xi)\Delta_{\xi\eta}g+g(\eta) \Delta_{\xi\eta}f$, hence 
$$\|fg\|^{\L}_{D(s,r)}\leq \|f\|_{D(s,r)}\|g\|^{\L}_{D(s,r)}+\|g\|_{D(s,r)}\|f\|^{\L}_{D(s,r)}.$$
Secondly, the operator $\Delta_{\xi\eta}$ commutes with the derivative in any variable.
  \end{proof}~
\subsection{The canonical transform}~\\[5pt]
In this Section we study  the Hamiltonian flow generated by a function $F \in \G^{\b,+}_{r,D(s-\s,r)}$  globally of degree 2, i.e. of degree 2 in the variables $z,\ov z$ and of degree 1 in the variable $y$. Namely, we consider the system
 \begin{equation}\label{System}
 \left\{
\begin{aligned}
&\big(\dot{\theta}(t),\dot{y}(t),\dot{z}(t),\dot{\z}(t)\big) = X_{F}\big(\big(\theta(t),y(t),z(t),\z(t)\big)\big) ,\\[3pt]
&\big(\theta(0),y(0),z(0),\z(0)\big)=\big(\theta^{0},y^{0},z^{0},\z^{0}\big).
\end{aligned}
\right.
\end{equation}

\begin{lemm}\label{lem.3}
 Let  $0<\s<s/3$ and  $F\in \G^{\b,+}_{r,D(s-\s,r)}$ with $F$ of degree 2.  Assume that  $ \dis \< F\>^{+}_{r,D(s-\s,r)} <C{\s}$. Then  the solution of the equation \eqref{System} with initial condition  $\big(\theta^{0},y^{0},z^{0},\z^{0}\big)\in D(s-3\s,\frac{r}4),$
 satisfies $\big(\theta(t),y(t),z(t),\z(t)\big)\in D(s-2\s,\frac{r}2)$ for all $0\leq t\leq 1$, and we have the estimates
 
  \begin{equation}\label{der.y.z}
  \sup_{0\leq t\leq 1} \Big|  \frac{\partial y_{k}(t)}{\partial w_{j}^{0}}\Big|\leq \frac{Cr\<F\>^{+}_{r,D(s-\s,r)}}{\s j^{\b}}\quad \text{with}\quad w^{0}_{j}=z^{0}_{j} \mbox{ or }\z^{0}_{j},  
  \end{equation}
  
    \begin{equation}\label{der.z.z}
  \sup_{0\leq t\leq 1} \Big|  \frac{\partial w_{k}(t)}{\partial w_{j}^{0}}\Big| \leq  \frac{C\<F\>^{+}_{r,D(s-\s,r)}}{(jk)^{\b}(1+|j-k|)}+\delta_{jk} \quad \text{with}\quad w_{k}=z_{k}  \mbox{ or }\z_{k},\;\;w^{0}_{j}=z^{0}_{j}  \mbox{ or }\z^{0}_{j},
 \end{equation}
\begin{equation}\label{der.y.y}
  \sup_{0\leq t\leq 1}\Big|  \frac{\partial y_{k}(t)}{\partial y_{j}^{0}}\Big|\leq \frac{C  \< F\>^{+}_{r,D(s-\s,r)}}{\s}+\delta_{jk},
   \end{equation}
  \begin{equation}\label{der.zz}
  \sup_{0\leq t\leq 1}\Big|  \frac{\partial^{2} y_{k}(t)}{\partial w_{j}^{0}\partial w_{i}^{0}}\Big|\leq  \frac{C\<F\>^{+}_{r,D(s-\s,r)}}{\s (ij)^{\b}(1+|i-j|)} \quad \text{with}\quad w^{0}_{i}=z^{0}_{i}  \mbox{ or }\z^{0}_{i},\;\;w^{0}_{j}=z^{0}_{j}  \mbox{ or }\z^{0}_{j}.
   \end{equation}

\end{lemm}
Before we turn to the proof of Lemma \ref{lem.3}, we introduce a space of infinite dimensional matrices, with decaying coefficients.\\[2pt]
Let $\|\cdot\|$ be any submultiplicative norm on $\mathcal{M}_{2,2}(\C)$, the space of the $2\times2$ complex matrices.
  For $\beta>0$, we say that $B\in \mathcal{M}_{s}^{\beta,+}$ if $\<\<\,B\,\>\>^{+}_{\b,s}<\infty$, where the norm $\<\<\,\cdot\,\>\>^{+}_{\b,s}$ is given by the condition\footnote{This means that $\<\<\,\cdot\,\>\>^{+}_{\beta,s}$ is the smallest real number which satisfies the mentioned conditions : this defines a norm.   }

   \begin{equation*}
\sup_{\xi\in \Pi}\sup_{|\text{Im}\, \theta|<s}  \|B_{jl}\|   \leq \frac{\<\<\,B\,\>\>^{+}_{\b,s}}{(jl)^{\b}(1+|j-l|)},\quad \forall \,j,l\geq1.
    \end{equation*}
Then we have the following result  
  \begin{lemm}\label{lemm.AB}
  Let $A,B \in  \mathcal{M}_{s}^{\beta,+}$. Then $AB \in  \mathcal{M}_{s}^{\beta,+}$ and 
   \begin{equation*}
   \<\<\,AB\,\>\>^{+}_{\b,s} \leq C\<\<\,A\,\>\>^{+}_{\b,s}\<\<\,B\,\>\>^{+}_{\b,s}.
  \end{equation*}
  \end{lemm}
  
  \begin{proof}
  For all $j,l\geq 1$, $\dis \big(AB\big)_{jl}=\sum_{k\geq 1}A_{jk}B_{kl}$. Since $\|\cdot\|$ is submultiplicative
    \begin{eqnarray}
    \|\big(AB\big)_{jl}\|&\leq &\sum_{k\geq 1}\|A_{jk}\|\|B_{kl}\| \nonumber \\
    &  \leq&  \frac{ \<\<\,A\,\>\>^{+}_{\b,s}\<\<\,B\,\>\>^{+}_{\b,s}}{(jl)^{\beta}}\sum_{k\geq 1}\frac{1}{k^{2\b}(1+|j-k|)(1+|l-k|)}. \label{A.B}
  \end{eqnarray}
Thanks to the triangle inequality,  for all $j,l\geq 1$, 
 \begin{equation*}
  \big\{k\geq 1\big\} \subset \big\{k\geq 1\;:\;|j-k|\geq \frac13|j-l|\big\}\bigcup  \big\{k\geq 1\;:\;|l-k|\geq \frac13|j-l|\big\},
  \end{equation*}
thus, by splitting the sum in \eqref{A.B} we obtain the desired result.
  
  \end{proof}

\begin{proof}[Proof of Lemma \ref{lem.3}]
Here we introduce the notations $Z_{j}=(z_{j},\ov{z}_{j})$ and $Z=(Z_{j})_{j\geq 1}$. Then $F$ reads 
 \begin{equation}\label{struct}
 F(\theta,y,Z)=b_{0}(\theta)+b_{1}(\theta)\cdot y+a(\theta)\cdot Z+\frac12 \big(A(\theta)Z\big)\cdot Z,
  \end{equation}
with 
 \begin{equation*}
 b_{0}(\theta)=F(\theta,0,0),\qquad  b_{1}(\theta)=\nabla_{y}F(\theta,0,0), \qquad  a(\theta)=\nabla_{Z}F(\theta,0,0),
  \end{equation*}
and $A=(A_{i,j})$ is the infinite matrix so that 
 \begin{equation}\label{Aij}
A_{i,j}(\theta)=
\begin{pmatrix}
\dis \frac{\partial^{2}F}{\partial z_{i}\partial z_{j}}(\theta,0,0) &  \dis \frac{\partial^{2}F}{\partial z_{i}\partial \ov{z}_{j}}(\theta,0,0)\\[12pt]
\dis  \frac{\partial^{2}F}{\partial \ov{z}_{i}\partial z_{j}}(\theta,0,0) &  \dis \frac{\partial^{2}F}{\partial \ov{z}_{i}\partial \ov{z}_{j}}(\theta,0,0)
\end{pmatrix}.
  \end{equation}
  Observe that $A$ is symmetric.\\[5pt]
  By \cite[Estimate (9)]{Poschel}, the flow $X_{F}^{t}$ exists for $0\leq t\leq 1$ and maps $D(s-3\s,\frac{r}4)$ into $D(s-2\s,\frac{r}2)$. Here we have to  give a precise description of  $X_{F}^{t}$ for  $0\leq t\leq 1$. This is possible thanks to the particular structure \eqref{struct} of F. \\
In the sequel we write $(\theta(t),y(t),Z(t))=X_{F}^{t}(\theta^{0},y^{0},Z^{0})$.\\[5pt]
$\spadesuit$ To begin with, the equation for $\theta$ reads 
 \begin{equation}\label{eq.theta}
 \dot{\theta}(t)=\nabla_{y}F(\theta,0,0)= b_{1}(\theta),\quad \theta(0)=\theta^{0}.
  \end{equation}
Since $b_{1}$ is a smooth   function (see \eqref{def.F}), the $n$-dimensional system \eqref{eq.theta} admits a unique (smooth) local solution $\theta(t)$. By the work of J. P\"oschel, this solution exists until time $t=1$, and we have the bound 
 \begin{equation}\label{sup.theta}
\sup_{0\leq t\leq 1} |\text{Im}\;\theta(t)|<s-2\s,
 \end{equation} 
 (this can here be recovered by the usual bootstrap argument, using the smallness assumption on $F$).\\[5pt]
 $\spadesuit$ We now turn to the equation in $Z$. We have to solve 
 \begin{equation}\label{eq.Z}
 \dot{Z}(t)=J\nabla_{Z}F(\theta,y,Z)(t),\quad Z(0)=Z^{0},
  \end{equation}
where 
\begin{equation*}
J=\text{diag}\Big\{
\begin{pmatrix}
0 &  1\\ 
-1 & 0 
\end{pmatrix}\Big\}_{j\geq 1}.
 \end{equation*}
 Notice that by \cite[Estimate (9)]{Poschel} we already know that  
 \begin{equation}\label{est.Z}
 \sup_{0\leq t \leq 1} \|Z(t)\|_{\ell^{2}_{\Psi}}<\frac{r}2
  \end{equation}
  but we need to precise the behavior of $Z(t)$.\\
Since  $\theta=\theta(t)$ is known by the previous step, in view of \eqref{struct}, equation \eqref{eq.Z} reads 
 \begin{equation} \label{Zt}
 \dot{Z}(t)=b(t)+B(t)\cdot Z(t),\quad Z(0)=Z^{0},
  \end{equation} 
  where $b(t)=Ja(\theta(t))$ and $B(t)=JA(\theta(t))$.\\
We now iterate the integral formulation of the problem
 \begin{equation*} 
Z(t)=Z^{0}+\int_{0}^{t}\big(b(t_{1})+B(t_{1})\cdot Z(t_{1}) \big)\text{d}t_{1}, 
  \end{equation*}
and formally obtain
 \begin{equation}\label{Zt2}
Z(t)= b^{\infty}(t)+ \big(1+B^{\infty}(t)\big)Z^{0},
  \end{equation}
  where 
    \begin{equation}\label{binf}
  b^{\infty}(t)=\sum_{k\geq 1}\int_{0}^{t}\int_{0}^{t_{1}}\cdots \int_{0}^{t_{k-1}} \prod_{j=1}^{k-1}B(t_{j})b(t_{k})\text{d}t_{k}  \cdots \text{d}t_{2}\,\text{d}t_{1},
  \end{equation}  
  and 
   \begin{equation}\label{def.Binfini}
  B^{\infty}(t)=\sum_{k\geq 1}\int_{0}^{t}\int_{0}^{t_{1}}\cdots \int_{0}^{t_{k-1}} \prod_{j=1}^{k}B(t_{j})\text{d}t_{k}  \cdots \text{d}t_{2}\,\text{d}t_{1}.
  \end{equation}  
  By \eqref{Aij} and \eqref{sup.theta}, there exists $C>0$ so that   
   \begin{equation*}
  \sup_{0\leq t\leq 1}\|B(t)\|_{\ell^{2}_{\Psi}\to \ell^{2}_{\Psi}}\leq C,
  \end{equation*} 
  and thus, for all $0\leq t\leq 1$ the series  \eqref{binf} converges and
\begin{eqnarray}
  \|b^{\infty}(t)\|_{\ell^{2}_{\Psi}}&\leq & \sup_{0\leq t\leq 1}\|b(t)\|_{\ell^{2}_{\Psi}}\sum_{k\geq 1}C^{k-1}\int_{0}^{1}\int_{0}^{t_{1}}\cdots \int_{0}^{t_{k-1}}\text{d}t_{k}  \cdots \text{d}t_{2}\,\text{d}t_{1}\nonumber\\
  &\leq & \sup_{0\leq t\leq 1}\|b(t)\|_{\ell^{2}_{\Psi}}\sum_{k\geq 1}\frac{C^{k-1}}{k\,!}\nonumber\\
 &\leq &\sup_{0\leq t\leq 1}\|b(t)\|_{\ell^{2}_{\Psi}} \frac{\e^{C}-1}{C}\nonumber\\
 &\leq &C\sup_{0\leq t\leq 1}\|b(t)\|_{\ell^{2}_{\Psi}}.\label{2.46}
  \end{eqnarray}
  Similarly we have  uniformly in $0\leq t\leq 1$
    \begin{equation*}
 \| B^{\infty}(t)\|_{\ell^{2}_{\Psi}\to \ell^{2}_{\Psi}}\leq C.
  \end{equation*}  
  As a conclusion, the formula \eqref{Zt2} makes sense.

 Indeed, we need more precise estimates on $B^{\infty}$.   Recall that $B(t)=A(\theta(t))$, where $A$ is defined by \eqref{Aij}. Then by \eqref{Aij} and \eqref{sup.theta}, for all $0\leq t\leq 1$, $B(t)\in \mathcal{M}^{\b,+}_{s-\s}$ and $\dis \sup_{0\leq t\leq 1}\<\<\,B(t)\,\>\>^{+}_{\b,s-\s}\leq C \<\,F\,\>^{+}_{r,D(s-\s,r)}$. Hence by Lemma \ref{lemm.AB} and \eqref{def.Binfini}
    \begin{equation}\label{<B>}
    \<\<\,B^{\infty}\,\>\>^{+}_{\b,s-\s}\leq \e^{C\<\,F\,\>^{+}_{r,D(s-\s,r)}}-1\leq C \<\,F\,\>^{+}_{r,D(s-\s,r)}.
  \end{equation}
 $\spadesuit$ Finally we turn to the equation in $y$
   \begin{equation*} 
 \dot{y}(t)=-\nabla_{\theta}F(\theta,y,Z)(t),\quad y(0)=y^{0}.
\end{equation*}
We already know the functions $\theta(t)$ and $Z(t)$. Moreover as the function $F$ \eqref{struct} is linear in $y$, the previous $n-$dimensional system reads
    \begin{equation}\label{eq.y}
 \dot{y}(t)=f(t)+g(t)y(t),\quad y(0)=y^{0},
  \end{equation}
  with 
     \begin{equation*}
     f(t)=-\nabla_{\theta}b_{0}(\theta(t))+\nabla_{\theta}a(\theta(t))\cdot Z(t)+\frac12\big(\nabla_{\theta}A(\theta(t))Z(t)\big)\cdot Z(t),
  \end{equation*}
  and 
   \begin{equation*}
   g(t)= -\nabla_{\theta}b_{1}(\theta(t)) =-\nabla_{\theta}\nabla_{y}F(\theta,0,0). 
     \end{equation*}
We can solve the equation \eqref{eq.y} with the same techniques as the equation \eqref{eq.Z}. In fact we have formally
 \begin{equation}\label{yt}
 y(t)=f^{\infty}(t)+\big(1+g^{\infty}(t)\big)y^{0},
  \end{equation}
  where 
    \begin{equation}\label{def.finf}
  f^{\infty}(t)=\sum_{k\geq 1}\int_{0}^{t}\int_{0}^{t_{1}}\cdots \int_{0}^{t_{k-1}} \prod_{j=1}^{k-1}g(t_{j})f(t_{k})\text{d}t_{k}  \cdots \text{d}t_{2}\,\text{d}t_{1},
  \end{equation}  
and 
  \begin{equation*} 
  g^{\infty}(t)=\sum_{k\geq 1}\int_{0}^{t}\int_{0}^{t_{1}}\cdots \int_{0}^{t_{k-1}} \prod_{j=1}^{k}g(t_{j})\text{d}t_{k}  \cdots \text{d}t_{2}\,\text{d}t_{1}.
  \end{equation*}  
By \eqref{sup.theta} and the Cauchy formula 
  \begin{equation*}
  \sup_{0\leq t\leq 1}\|g(t)\|\leq \frac{C}{\s}\max_{1\leq j\leq n}\Big\|\frac{\partial F}{\partial y_{j}}\Big\|_{D(s-\s,r)}\leq \frac{C \<\,F\,\>^{+}_{r,D(s-\s,r)}} {\s},
  \end{equation*}
  and  similarly to \eqref{2.46} we have  for all $0\leq t\leq 1$
\begin{equation*}
  |f^{\infty}(t)|\leq  C \sup_{0\leq t\leq 1}|f(t)|,
  \end{equation*}
  and 
  \begin{equation}\label{norm.inf.g}
  \|g^{\infty}(t)\|\leq \frac{C \<\,F\,\>^{+}_{r,D(s-\s,r)}} {\s} ,
  \end{equation}
  which shows the convergence of the series defining \eqref{yt}.\\[5pt]
 $\spadesuit$ It remains to show the estimates on the solutions of \eqref{System}.\\[5pt]
   $\bullet$ First we prove \eqref{der.z.z}. By \eqref{Zt},
   \begin{equation*}
 \nabla_{Z_{j}^{0}}Z_{k}(t)=\left(\begin{array}{cc}1 & 0  \\0 & 1 \end{array}\right)\delta_{kj}+B^{\infty}_{kj}(t),
  \end{equation*}
 therefore by \eqref{<B>}, for $k\neq j$ we have 
  \begin{equation}\label{ZZ}
 \| \nabla_{Z_{j}^{0}}Z_{k}(t)\| \leq  \frac{C\<F\>^{+}_{r,D(s-\s,r)}}{(jk)^{\b}(1+|j-k|)}, \quad \text{and}\quad \| \nabla_{Z_{j}^{0}}Z_{j}(t)\| \leq 1,
 \end{equation}
  which was the claim.\\[3pt]
  $\bullet$ We prove \eqref{der.y.y}. By \eqref{yt} we have 
 \begin{equation*} 
 y_{k}(t)=f_{k}^{\infty}(t)+y_{k}^{0}+\sum_{1\leq j\leq n}g_{jk}^{\infty}(t)y_{j}^{0},
  \end{equation*}
hence $\dis  \frac{\partial y_{k}}{\partial y_{j}^{0}}=\delta_{jk}+g_{jk}^{\infty}(t)$ and the claim follows from \eqref{norm.inf.g} ($f^{\infty}$ does not depend on $y^{0}$).\\[3pt]
  $\bullet$ We prove \eqref{der.y.z}. Since  $g$ and $g^{\infty}$ do not depend on $Z$, from \eqref{yt} we deduce that $\dis \frac{\partial y}{\partial z_{j}^{0}}=\frac{\partial f^{\infty}}{\partial z_{j}^{0}}$. \\
  Now by definition \eqref{def.finf} of $f^{\infty}$, we get that for all $0\leq t\leq 1$
   \begin{equation}\label{...}
 \Big|\frac{\partial y(t)}{\partial z_{j}^{0}}\Big|=\Big|\frac{\partial f^{\infty}(t)}{\partial z_{j}^{0}}\Big|\leq \big|\nabla_{Z^{0}_{j}}f^{\infty}(t)\big|\leq  C \sup_{0\leq t\leq 1}| \nabla_{Z^{0}_{j}}f(t)|.
  \end{equation}
  For all $1\leq l \leq n$, we compute 
 \begin{equation}\label{3.54b}
  \nabla_{Z_{k}}f_{l}(t)=\partial_{\theta_{l}}a_{k}(\theta(t))+\sum_{i\geq 1}\partial_{\theta_{l}}A_{ki}(\theta(t))Z_{i}(t).
  \end{equation}
  As $a_{k}(\theta)=\nabla_{Z_{k}}F(\theta,0,0)$, with the Cauchy formula we deduce 
  \begin{equation*}
  \sup_{0\leq t\leq 1}\big|\partial_{\theta_{l}}a_{k}(\theta(t))\big|\leq \frac{C}{\s}\big\|\nabla_{Z_{k}}F\big\|_{D(s-\s,r)}\leq \frac{Cr\<F\>^{+}_{r,D(s-\s,r)} }{\s k^{1+\beta}}.
  \end{equation*}
  Similarly with \eqref{Aij},
    \begin{equation*}
  \sup_{0\leq t\leq 1}\big|\partial_{\theta_{l}}A_{ki}(\theta(t))\big|\leq  \frac{C\<F\>^{+}_{r,D(s-\s,r)} }{\s (ik)^{\beta}(1+|i-k|)}.
  \end{equation*}
 Inserting the two previous estimates in \eqref{3.54b}, we obtain using \eqref{est.Z} and the Cauchy-Schwarz inequality
   \begin{eqnarray}
| \nabla_{Z_{k}} f_{l}(t)|&\leq &  \frac{C}{\s}\frac{\<F\>^{+}_{r,D(s-\s,r)}}{k^{\beta}}\Big(r+\sum_{i\geq 1}\frac{|Z_{i}|}{i^{\b}(1+|k-i|)}\Big)\nonumber\\
&\leq &  \frac{Cr\<F\>^{+}_{r,D(s-\s,r)}}{\s k^{\b}}\label{est.ZZ}.
\end{eqnarray}
  Since 
   $ \dis  \nabla_{Z^{0}_{j}}f_{l}(t)=\sum_{k\geq 1}  \big( \nabla_{Z^{0}_{j}}Z_{k}(t) \big)  \nabla_{Z_{k}}f_{l}(t)$,
 from \eqref{ZZ} and \eqref{est.ZZ} we deduce 
     \begin{eqnarray*}
    | \nabla_{Z^{0}_{j}}f_{l}(t)|&\leq& \sum_{k\geq 1} \|\nabla_{Z^{0}_{j}}Z_{k}(t)\|  \| \nabla_{Z_{k}}f_{l}(t)\|\nonumber	\\
    &\leq& \frac{Cr\<F\>^{+}_{r,D(s-\s,r)}}{\s j^{\b}}\Big(  \sum_{k\geq 1} \frac{1}{k^{2\b}(1+|j-k|)}+1\Big)\\
       &\leq& \frac{Cr\<F\>^{+}_{r,D(s-\s,r)}}{\s j^{\b}},
  \end{eqnarray*}
  and together with \eqref{...}, we get that for all $j\geq 1$
  \begin{equation*}
 \sup_{0\leq t\leq 1}\Big|\frac{\partial y(t)}{\partial z_{j}^{0}}\Big|\leq \frac{Cr\<F\>^{+}_{r,D(s-\s,r)}}{\s j^{\b}}.
   \end{equation*}
 $\bullet$ It remains to show \eqref{der.zz}. first we have
 \begin{equation*}
 \Big|\frac{\partial y(t)}{\partial z_{i}^{0}\partial z_{j}^{0}}\Big|\leq  \Big|\nabla_{Z^{0}_{i}}\nabla_{Z^{0}_{j}}f^{\infty}(t)\Big|\leq C \sup_{0\leq t\leq 1} \Big|\nabla_{Z^{0}_{i}}\nabla_{Z^{0}_{j}}f(t)\Big| .
   \end{equation*}
   Then  from the very definition of $f$, $\nabla_{Z^{0}_{i}}\nabla_{Z^{0}_{j}}f(t)=\nabla_{\theta}A_{ij}(\theta(t))$, and using the Cauchy estimate in $\theta$ we get,
     \begin{equation*}
      \Big|\frac{\partial y(t)}{\partial z_{i}^{0}\partial z_{j}^{0}}\Big|\leq \frac{C\<F\>^{+}_{r,D(s-\s,r)}}{\s (ij)^{\b}(1+|i-j|)},
  \end{equation*}
 which was the claim.
  \end{proof}
 In the next result, we denote by $|\cdot|^{\L}$ the Lipschitz norm
 $$|f|^{\L}=\sup_{\substack{\xi,\eta\in \Pi\\ \xi\neq \eta}}\frac{|f(\xi)-f(\eta)|}{|\xi-\eta|}.$$
 We have an analogous result to Lemma \ref{lem.3} with Lipschitz norms.
 \begin{lemm}\label{lem.3*}
 Under the assumptions of Lemma \ref{lem.3} and the condition \linebreak[4] $\<F\>^{+,\L}_{r,D(s-\s,r)}\leq C\s$ the solution of \eqref{System} satisfies moreover 
 \begin{eqnarray*} 
  \sup_{0\leq t\leq 1} \Big|  \frac{\partial y_{k}(t)}{\partial w_{j}^{0}}\Big|^{\L}&\leq &\frac{Cr\<F\>^{+,\L}_{r,D(s-\s,r)}}{\s j^{\b}}\quad \text{with}\quad w^{0}_{j}=z^{0}_{j} \mbox{ or }\z^{0}_{j},  \\
  \sup_{0\leq t\leq 1} \Big|  \frac{\partial w_{k}(t)}{\partial w_{j}^{0}}\Big|^{\L}& \leq&  \frac{C\<F\>^{+,\L}_{r,D(s-\s,r)}}{(jk)^{\b}(1+|j-k|)} \quad \text{with}\quad w_{k}=z_{k} \mbox{ or }\z_{k},\;\;w^{0}_{j}=z^{0}_{j} \mbox{ or }\z^{0}_{j},\\
  \sup_{0\leq t\leq 1}\Big|  \frac{\partial y_{k}(t)}{\partial y_{j}^{0}}\Big|^{\L}&\leq& \frac{C  \< F\>^{+,\L}_{r,D(s-\s,r)}}{\s},\\
 \sup_{0\leq t\leq 1}\Big|  \frac{\partial^{2} y_{k}(t)}{\partial w_{j}^{0}\partial w_{i}^{0}}\Big|^{\L}&\leq & \frac{C\<F\>^{+,\L}_{r,D(s-\s,r)}}{\s (ij)^{\b}(1+|i-j|)} \quad \text{with}\quad w^{0}_{i}=z^{0}_{i} \mbox{ or }\z^{0}_{i},\;\;w^{0}_{j}=z^{0}_{j} \mbox{ or }\z^{0}_{j}.
   \end{eqnarray*}
\end{lemm}

\begin{proof}
We won't detail the proof, since it is tedious and similar to the proof of Lemma \ref{lem.3}. First we define the space $\mathcal{M}^{\beta,+,\L}_{s}$ with norm $\<\<\,\cdot\,\>\>^{+,\L}_{\beta,s}$ similarly to $\mathcal{M}^{\beta,+}_{s}$, but with a Lipschitz norm in $\xi$. Then we have 
$\dis \<\<AB\>\>^{+,\L}\leq C\big(\<\<A\>\>^{+,\L}\<\<B\>\>^{+}+\<\<B\>\>^{+,\L}\<\<A\>\>^{+}\big)$. Then one can follow the proof of Lemma \ref{lem.3} and use that the different norms (say $\|\,\cdot\,\|$) which appear satisfy $\|fg\|^{\L}\leq C\big(\|f\|^{\L}\|g\|+\|f\|\|g\|^{\L}\big)$.

\end{proof}

To conclude this section, we state a result which shows that the Lie transform associated to a quadratic  function, is also quadratic. This will be crucial in the proof of Theorem \ref{thmKAM} (see Section \ref{Proof}).
 
 \begin{coro}\label{coro_sol}
 The symplectic application $X_{F}^{1}$ reads 
   \begin{equation*} 
 \left(\begin{array}{c}\theta \\y \\Z\end{array}\right)\longmapsto\left(\begin{array}{l} K(\theta) \\L(\theta,Z) + M(\theta)Z+S(\theta)y \\T(\theta)+U(\theta)Z\end{array}\right)
 \end{equation*}
 where $L(\theta,Z)$ is quadratic in $Z$, $M(\theta)$ and $ U(\theta)$  are bounded linear operators from $\ell^{2}_{\Psi}\times\ell^{2}_{\Psi}$ into itself and $S(\theta)$ is a bounded linear map from $\R^n$ to $\R^n$.
 \end{coro}
 
 \begin{proof}
 The claim follows from the proof of Lemma \ref{lem.3}. The structure of $Z(1)$ follows from \eqref{Zt}, while the structure of $y(1)$ comes from \eqref{yt} and \eqref{def.finf}.  
 \end{proof}
 ~
 \subsection{Composition estimates}~\\[5pt]
In this section we study the new Hamiltonian obtained after composition with the canonical transformation $X_{F}^{1}$.
  \begin{prop}\label{prop.comp}
Let $0<\eta<1/8$ and  $0<\s<s$,  $R\in \G^{\b}_{\eta r,D(s-2\s,4\eta r)}$ and  $F\in \G^{\b,+}_{r,D(s-\s,r)}$ with $F$ of degree 2.  Assume that  $ \dis \< F\>^{+}_{r,D(s,r)}+  \< F\>^{+,\L}_{r,D(s,r)}<C{\s}$. Then  
$R \circ X^{1}_{F} \in \G^{\b}_{\eta r,D(s-5\s, \eta r)}$
and we have the estimates 
\begin{equation}\label{est.comp}
\<\,R\circ X^{1}_{F}\,\>_{\eta r,D(s-5\s,\eta r)} \leq  C\<\,R\,\>_{\eta r,D(s-2\s,4 \eta r)},
  \end{equation}
  \begin{equation*}
\<\,R\circ X^{1}_{F}\,\>^{\L}_{\eta r,D(s-5\s,\eta r)} \leq  C\(\<\,R\,\>_{\eta r,D(s-2\s,4\eta r)}+\<\,R\,\>^{\L}_{\eta r,D(s-2\s,4\eta r)}\).
  \end{equation*}
\end{prop}

\begin{proof} The proof of the first estimate relies on Lemma \ref{lem.3}. We omit the proof of the second, which is similar using the estimates of Lemma \ref{lem.3*} instead. \\
In the sequel, we use the notation $(\theta,y,z,\bar z) =X^{1}_{F}(\theta^0,y^0,z^0,\bar z^0)$.\\ 
$\spadesuit$ Since $X_{F}^{1}$ maps $D(s-3\s,\frac{r}4)$ into $D(s-2\s,\frac{r}2)$, it is clear that 
 \begin{equation}\label{Est.o.1}
 \| R\circ X^{1}_{F}\|_{D(s-5\s,\eta r)}\leq C\<R\>_{\eta r,D(s-2\s,4\eta r)}. 
 \end{equation}
$\spadesuit$ By the Leibniz rule, for all $1\leq j\leq n$
 \begin{equation*}
 \frac{\partial( R\circ X^{1}_{F})}{\partial y^{0}_{j}}=\sum_{k=1}^{n}  \frac{\partial R(X_{F}^{1})}{\partial y_{k}} \frac{\partial y_{k}}{\partial y^{0}_{j}},
 \end{equation*}
and by \eqref{der.y.y} we deduce 
 \begin{equation}\label{Est.o.2}
\Big \| \frac{\partial( R\circ X^{1}_{F})}{\partial y^{0}_{j}}\Big\|_{D(s-5\s,\eta r)}\leq C\<R\>_{\eta r,D(s-2\s,4\eta r)}. 
 \end{equation}
$\spadesuit$ For $j\geq 1$, the derivative in $z^{0}_{j}$ reads 
 \begin{multline*}
 \frac{\partial( R\circ X^{1}_{F})}{\partial z^{0}_{j}}=\\
 \begin{aligned}
& \sum_{k=1}^{n}  \frac{\partial R(X_{F}^{1})}{\partial y_{k}} \frac{\partial y_{k}}{\partial z^{0}_{j}}+\sum_{k\geq 1}\Big( \frac{\partial R(X_{F}^{1})}{\partial z_{k}}\frac{\partial z_{k}}{\partial z^{0}_{j}}+\frac{\partial R(X_{F}^{1})}{\partial \z_{k}}\frac{\partial \z_{k}}{\partial z^{0}_{j}}\Big).
 \end{aligned}
 \end{multline*}
Therefore, thanks to \eqref{der.y.z} and \eqref{der.zz} we get
\begin{multline}\label{Est.o.3}
\Big\| \frac{\partial( R\circ X^{1}_{F})}{\partial z^{0}_{j}}\Big\|_{D(s-5\s,\eta r)}\leq \\
 \begin{aligned}
 & \leq \sum_{k=1}^{n}  \Big\|\frac{\partial R(X_{F}^{1})}{\partial y_{k}} \Big\|_{D(s-5\s,\eta r)} \Big|\frac{\partial y_{k}}{\partial z^{0}_{j}}\Big|+\sum_{k\geq 1}\Big\|\nabla_{Z_{k}}R(X_{F}^{1})\Big\|_{D(s-5\s,r)} \Big|\frac{\partial Z_{k}}{\partial z^{0}_{j}}\Big|\\
& \leq \frac{C}{j^{\beta}}\<R\>_{\eta r,D(s-2\s,4\eta r)}\Big(1+     \sum_{k\geq 1}\frac{1}{k^{2\b}(1+|j-k|)}     \Big) \\
& \leq \frac{C}{j^{\beta}}\<R\>_{\eta r,D(s-2\s,4\eta r)}. 
 \end{aligned}
 \end{multline}
 $\spadesuit$ We now estimate $\dis \Big\| \frac{\partial^{2}( R\circ X^{1}_{F})}{\partial z^{0}_{i}\partial z^{0}_{j}}\Big\|_{D(s-5\s,\eta r)}$ for $i,j\geq 1$. By the Leibniz rule, the result will follow from the next estimations. \\[5pt]
 \indent $\bullet$ Using the Cauchy estimate in $y_{l}$ and  \eqref{der.y.z}
\begin{equation*}
\Big\|  \sum_{1\leq k,l\leq n}   \frac{\partial^{2} R(X_{F}^{1})}{\partial y_{k}\partial y_{l}}  \frac{\partial y_{k}}{\partial z^{0}_{i}} \frac{\partial y_{l}}{\partial z^{0}_{j}} \Big\|_{D(s-5\s,\eta r)}\leq 
 \frac{C \<R\>_{\eta r,D(s-2\s,4\eta r)}}{(ij)^{\b}}.
\end{equation*}
 \indent $\bullet$  By \eqref{der.zz}
\begin{equation*}
\Big\|  \sum_{1\leq k\leq n} \frac{\partial R(X_{F}^{1})}{\partial y_{k}} \frac{\partial^{2}y_{k}}{\partial z^{0}_{i}\partial z^{0}_{j}} \Big\|_{D(s-5\s,\eta r)}\leq \frac{C \<R\>_{\eta r,D(s-2\s,4\eta r)}}{(ij)^{\b}}.
\end{equation*}
 \indent $\bullet$ By \eqref{der.z.z}
\begin{equation*}
\Big\|  \sum_{k,l\geq 1 }   \frac{\partial^{2} R(X_{F}^{1})}{\partial z_{k}(t)\partial z_{l}}  \frac{\partial z_{k}}{\partial z^{0}_{i}} \frac{\partial z_{l}}{\partial z^{0}_{j}} \Big\|_{D(s-5\s,\eta r)}\leq \frac{C \<R\>_{\eta r,D(s-2\s,4\eta r)}}{(ij)^{\b}}.
 \end{equation*}
\indent $\bullet$  Using the Cauchy estimate in $z_{k}$, \eqref{der.y.z} and \eqref{der.z.z} we get
\begin{equation*}
\Big\|  \sum_{ \substack{k\geq 1\\1\leq l\leq n}} \frac{\partial^{2} R(X_{F}^{1})}{\partial z_{k}\partial y_{l}} \frac{\partial y_{l}}{\partial z^{0}_{i}} \frac{\partial z_{k}}{\partial z^{0}_{j}}  \Big\|_{D(s-5\s,\eta r)}\leq  \frac{C \<R\>_{\eta r,D(s-2\s,4\eta r)}}{(ij)^{\b}}.
\end{equation*}
All these estimates yield 
\begin{equation}\label{Est.o.4}
\Big\| \frac{\partial^{2}( R\circ X^{1}_{F})}{\partial z^{0}_{i}\partial z^{0}_{j}}\Big\|_{D(s-5\s,\eta r)}\leq \frac{C \<R\>_{\eta r,D(s-2\s,4\eta r)}}{(ij)^{\b}}.
\end{equation}~\\
Finally, \eqref{est.comp} follows from \eqref{Est.o.1}, \eqref{Est.o.2}, \eqref{Est.o.3} and \eqref{Est.o.4}.\\[4pt]
\end{proof}~
 
 \subsection{Approximation estimates}~\\[5pt]
 Recall that  the notation $\|\cdot\|^{*}$ (respectively $\<\,\cdot\,\>^{*}$\,) stands either for $\|\cdot\|$ or $\|\cdot\|^{\L}$ (respectively $\<\,\cdot\,\>$ or $\<\,\cdot\,\>^{\L}$\,).\\[5pt]
 First we  recall some approximation results \cite[Estimate (7)]{Poschel}, which show that the second order approximation of $P$ can be controlled by $P$, and that $P-R$ is small when we contract the domain (this contraction is governed by the new parameter $\eta$):
\begin{lemm}[\cite{Poschel}]\label{lem.taylor.P}
Let $P$ satisfy Assumption \ref{AS3}  and consider its Taylor approximation $R$ of the form \eqref{taylor}. Then there exists $C>0$ so that for all $\eta>0$
 \begin{equation*}
 \|X_{R} \|^{*}_{r,D(s,r)}\leq C  \|X_{P}\|^{*}_{r,D(s,r)},\quad \text{and}\quad   \|X_{P}-X_{R} \|^{*}_{\eta r,D(s,4\eta r)}\leq C\eta   \|X_{P} \|^{*}_{r,D(s,r)}.
  \end{equation*}
\end{lemm}
We have an analogous result for the norm $\<\,\cdot\,\>_{r,D(s,r)}$.
\begin{lemm}\label{lem.taylor}
Let $P\in \G^{\b}_{r,D(s,r)}$ and consider its Taylor approximation $R$ of the form \eqref{taylor}. Then there exists $C>0$ so that for all $\eta>0$
 \begin{equation*}
 \<R\>^{*}_{r,D(s,r)}\leq C  \<P\>^{*}_{r,D(s,r)},
  \end{equation*}
  and 
   \begin{equation*}
 \<P-R\>^{*}_{\eta r,D(s,4\eta r)}\leq C\eta  \<P\>^{*}_{r,D(s,r)}.
  \end{equation*}
\end{lemm}

\begin{proof}
$\bullet$
We first prove the second estimate. Define the one variable function $f(t)=P(\theta, t^{2}y,tz,t\z)$. Then by the Taylor formula, there exists $0<t_{0}<1$ so that
\begin{equation*}
f(1)=f(0)+f'(0)+\frac12f''(0)+\frac{1}6f^{(3)}(t_{0}),
\end{equation*} 
which reads
\begin{eqnarray*}
P(\theta,y,z,\z)-R(\theta,y,z,\z)&=&\frac{1}6f^{(3)}(t_{0})\\
&=&\mathcal{O}\Big( \,   z^{3} \frac{\partial^{3} P}{\partial z^{3}},\,yz\frac{\partial^{2} P}{\partial y\partial z}, \,y^{2}\frac{\partial^{2} P}{\partial y^{2}}\,\Big).
\end{eqnarray*}
Using the Cauchy estimates in $z$ or in  $y$, we obtain
\begin{equation*}
\|P-R\|_{D(s,4\eta r)}\leq C \eta\, (\eta r)^{2}\<P\>_{r,D(s, r)}.
\end{equation*}
The  estimates of the derivatives are obtained by the same method, with the adequate choice of the function $f$. A derivative in $z$ costs $\eta$ and a derivative in $y$ costs $\eta^{2}$.\\
 It is then also clear that we have  $ \<P-R\>^{\L}_{\eta r,D(s,4\eta r)}\leq C\eta  \<P\>^{\L}_{r,D(s,r)}$.\\[5pt]
 $\bullet$ The inequality  $ \<R\>^{*}_{ r,D(s, r)}\leq C  \<P\>^{*}_{r,D(s,r)}$ is a consequence of the previous point with $\eta=1$.
\end{proof}

\section{The KAM step}\label{Sect.KAM}
~\\[5pt]
Let $N$ be a Hamiltonian in normal form as in \eqref{Ham.N}, which reads in the variables $(\theta, y,z,\z)$, 
\begin{equation*}
N=\sum_{1\leq j\leq n}\om_{j}(\xi)+\sum_{j\geq 1}\Omega(\xi)z_{j}\z_{j},
\end{equation*}
and suppose that the Assumptions \ref{AS1} and \ref{AS2} are satisfied.\\
Consider a perturbation $P$ which satisfies Assumptions \ref{AS3} and \ref{AS4} for some  $r,s>0$. Then chose $0<\eta <1/8$, $0<\s<s,$ and   assume that 
\begin{equation}\label{Init}
\<P\>_{r,D(s, r)}+\|X_{P}\|_{ r,D(s, r)}+\frac{\alpha}{M}\Big(  \<P\>^{\L}_{r,D(s, r)}+\|X_{P}\|^{\L}_{ r,D(s, r)}  \Big )\leq \frac{\alpha \s^{t+1} \eta^{2}}{c_{0}},
\end{equation}
where $t$ is given by Lemmas \ref{lem.Poschel1} and \ref{lem.1},  $c_{0}$ is a large constant depending only on $n$ and $\tau$ (see \cite[Estimate (6)]{Poschel}.) \\
Thus, by Lemmas \ref{lem.Poschel1} and \ref{lem.taylor.P}, the solution $F$ of the homological equation \eqref{homol} satisfies 
\begin{equation*}
 \|X_{F}\|^{*}_{r,D(s-\s, r)}\leq \frac{C}{\alpha \s^{t}} \|X_{P}\|^{\L}_{r,D(s-\s, r)}\leq \s \eta^{2}.
\end{equation*}
Similarly, by Lemmas \ref{lem.1} and \ref{lem.taylor}
\begin{equation*}
 \<F\>^{+,*}_{r,D(s-\s, r)}\leq \frac{C}{\alpha \s^{t}}\<P\>^{\L}_{r,D(s,r)}\leq \s \eta^{2},
\end{equation*}
so that the hypothesis  Lemma \ref{lem.3} are fulfilled. \\
We  use the notations of Section \ref{Strategy}. \\[3pt]

\subsection{Estimates on the new error term}~\\[5pt]
We estimate the new error term $P_+$ given by \eqref{newP}.
\begin{lemm} Assume \eqref{Init}. Then there exists $C>0$ (independent of $\eta$ and $\s$) so that for all $\dis 0\leq \lambda \leq \frac{\alpha}M$
\begin{multline*}
\<P_{+}\>^{\lambda}_{\eta r,D(s-5\s,\eta r)}+\|X_{P_{+}}\|^{\lambda}_{\eta r,D(s-5\s,\eta r)}\leq \\
\frac{C}{\alpha \s^{t} \eta^{2}}\Big(\<P\>^{\lambda}_{r,D(s, r)}+\|X_{P}\|_{r,D(s,r)}^{\lambda}\Big)^{2}+C\eta \Big(\<P\>^{\lambda}_{r,D(s,r)}+\|X_{P}\|_{r,D(s,r)}^{\lambda}\Big)
  \end{multline*}
\end{lemm} 

\begin{proof}
By \cite[Estimate (13)]{Poschel}, we already have 
\begin{equation}\label{EstP+}
\|X_{P_{+}}\|^{\lambda}_{\eta r,D(s-5\s,\eta r)}\leq \frac{C}{\alpha \s^{t} \eta^{2}}\big(\|X_{P}\|_{r,D(s,r)}^{\lambda}\big)^{2}+C\eta \|X_{P}\|_{r,D(s,r)}^{\lambda}.
  \end{equation}
It remains to prove a similar estimate for the $\<,\>$ norm.\\
  By Lemmas \ref{lem.3} and \ref{lem.taylor}
  \begin{equation*}
  \<(P-R)\circ X_{F}^{1}\>^{\lambda}_{\eta r,D(s-5\s,\eta r)}\leq C   \<P-R\>^{\lambda}_{\eta r,D(s-2\s,4\eta r)}\leq C\eta \<P\>^{\lambda}_{r,D(s, r)}.
  \end{equation*}
  Then by Lemma \ref{lem.3} again
   \begin{eqnarray*}
  \<\int_{0}^{1}\cp{R(t),F}\circ X^{t}_{F}\,\text{d}t\>^{\lambda}_{\eta r,D(s-5\s,\eta r)}&\leq &C   \int_{0}^{1}\<\cp{R(t),F}\circ X^{t}_{F}\>^{\lambda}_{\eta r,D(s-5\s,\eta r)}\text{d}t\\
  &\leq& C \<\cp{R(t),F}\>^{\lambda}_{\eta r,D(s-2\s,4\eta r)}.
  \end{eqnarray*}
Since $R\in \G^{\b}_{r,D(s,r)}$ and $F\in \G^{\b,+}_{r,D(s-\s,r)}$ are both of degree 2 we can apply Lemma \ref{lem.2} and write 
 \begin{equation*}
  \<\int_{0}^{1}\cp{R(t),F}\circ X^{t}_{F}\,\text{d}t\>^{\lambda}_{\eta r,D(s-5\s,\eta r)} \leq \frac{C}{\s}  \<R\>^{\lambda}_{\eta r,D(s,\eta r)} \,\<F\>^{+,\lambda}_{\eta r,D(s-\s,\eta r)}.
  \end{equation*}
Finally by  Lemmas \ref{lem.1} and \ref{lem.taylor}
  \begin{equation*}
  \<R\>^{\lambda}_{\eta r,D(s,\eta r)}\, \<F\>^{+,\lambda}_{\eta r,D(s-\s,\eta r)}\leq \frac{C}{\alpha \s^{t}} \Big(\<R\>^{\lambda}_{\eta r,D(s,\eta r)}\Big)^{2} \leq    \frac{C}{\alpha\s^{t}\eta^{2}}  \Big(\<P\>^{\lambda}_{r,D(s, r)}\Big)^{2},
  \end{equation*}
  where we used that $\dis   \<\,\cdot \,\>_{\eta r,D(s,\eta r)} \leq  \eta^{-2}\<\,\cdot\, \>_{r,D(s, r)}$. 
  Putting the previous estimates together, we  complete the proof.
  \end{proof}~

\subsection{Estimates on the frequencies}\label{Sect.4.2.}~\\[5pt]

We  turn to the new frequencies given by \eqref{newfreq}. 
 \begin{lemm}
There exists $K>10$ and  $\alpha_{+}>0$ so that 
 \begin{equation*}
\big |k\cdot \omega^{+}(\xi)+l\cdot \Omega^{+}(\xi)\big|\geq \alpha_{+}\frac{\<l\>}{A_{k}}, \quad |k|\leq K, \quad |l|\leq 2.
 \end{equation*}
  \end{lemm}
  In fact $K$ can be made explicit, it depends on $n,\tau,c_{0}$ and on all the constants $C$.
 \begin{proof}
On the one hand, since $\dis \wh{\om}_{j}(\xi)=\frac{\partial \wh{N}}{\partial y_{j}}(0,0,0,0,\xi)$, by Lemma \ref{lem.taylor.P} we deduce that 
\begin{equation*}
|\wh{\om}|_{\Pi}\leq \sup_{D(s,r)\times \Pi}|\frac{\partial \wh{N}}{\partial y}|\leq   \|X_{\widehat{N}}\|_{r,D(s,r)}\leq C \|X_{R}\|_{r,D(s,r)}\leq C \|X_{P}\|_{r,D(s,r)}.
\end{equation*}
On the other hand, $\dis \wh{\Omega}_{j}(\xi)=\frac{\partial^{2} \wh{N}}{\partial z_{j}\partial \z_{j}}(0,0,0,0,\xi)$, thus 
\begin{equation}\label{key}
|\wh{\Omega}|_{2\beta,\Pi}\leq \sup_{D(s,r)\times \Pi}|\frac{\partial^{2} \wh{N}}{\partial z_{j}\partial \z_{j}}|j^{2\beta}\leq \<\wh{N}\>_{r,D(s-\s,r)}\leq C\<R\>_{r,D(s,r)}\leq C \<P\>_{r,D(s,r)},
\end{equation}
hence by the two previous estimates 
\begin{equation}\label{3.2}
|\wh{\om}|_{\Pi}+|\wh{\Omega}|_{2\beta,\Pi}\leq C \( \|X_{P}\|_{r,D(s,r)}+ \<P\>_{r,D(s,r)}\).
\end{equation}
Similarly, for the Lipschitz norms we obtain
\begin{equation*}
|\wh{\om}|^{\L}_{\Pi}+|\wh{\Omega}|^{\L}_{2\beta,\Pi}\leq C \( \|X_{P}\|^{\L}_{r,D(s,r)}+ \<P\>^{\L}_{r,D(s,r)}\).
\end{equation*}
We  follow the analysis done in \cite{Poschel} to bound the small divisors and thanks to \eqref{3.2}
\begin{eqnarray*}
|\,k\cdot \wh{\om}+l\cdot \wh{\Omega}\,|&\leq&  |k|\<l\>\( |\wh{\om}|_{\Pi}+|\wh{\Omega}|_{2\beta,\Pi}\)\\
&\leq& C |k|\<l\>\( \|X_{P}\|_{r,D(s,r)}+ \<P\>_{r,D(s,r)}\).
\end{eqnarray*}
We now choose $\dis \wh{\alpha}\geq C_{0}K\max_{|k|\leq K}A_{k}( \|X_{P}\|_{r,D(s,r)}+ \<P\>_{r,D(s,r)})$ where $C_{0}$ is a large universal constant, and thanks to the estimate given by the frequencies before the iteration we get 
 \begin{equation*}
 |k\cdot \omega^{+}(\xi)+l\cdot \Omega^{+}(\xi)|\geq \alpha^{+}\frac{\<l\>}{A_{k}}, \quad |k|\leq K,
 \end{equation*}
with $\alpha_{+}=\alpha-\wh{\alpha}$. It remains to show that $\alpha^{+}>0$. This is done in \cite[Section 4]{Poschel}, and the proof still holds with the new norms.
\end{proof}

\begin{rema}\label{Rq4}
The key point in the previous proof is the estimate \eqref{key}, which shows that the perturbations of the external frequencies can be controlled by  $\<P\>_{r, D(s,r)}$. In the case of a smoothing perturbation $P$ (case $\ov{p}>p$ in \eqref{XP}), the norm $\<\cdot\>_{r, D(s,r)}$ is not needed (more precisely, the decay of the derivatives of $P$ is not needed), because we then have $|\wh{\Omega}|_{2\beta,\Pi}\leq \|X_{P}\|_{r,D(s,r)}$ with $\beta=(\ov{p}-p)/2$.
\end{rema}

\section{Iteration and convergence }\label{Sect.Conv}

In this section we are exactly in the setting of \cite{Poschel}, and we can make the same choice of the parameters in the iteration. We reproduce here the argument of J. P\"oschel.\\
\subsection{The iterative lemma}~\\[5pt]
Denote $P_{0}=P$ and $N_{0}=N$. Then at the $\nu-$th step  of the Newton scheme, we have a Hamiltonian $H_{\nu}=N_{\nu}+P_{\nu}$, so that the new error term $P_{\nu+1}$ is given by the formula \eqref{newP} and the new normal form $N_{\nu+1}$ is associated with the new frequencies given by \eqref{newfreq}.\\[5pt]
Let $c_{1}$ be twice the maximum of all constants obtained during the KAM step.\\
 Set $r_{0}=r$, $s_{0}=s$, $\alpha_{0}=\alpha$ and $M_{0}=M$.  For $\nu\geq 0$ and $\dis \kappa=4/3$ set 
\begin{equation*}
\alpha_{\nu}=\frac{\alpha_{0}}{2}(1+2^{-\nu}), \quad M_{\nu}=M_{0}(2-2^{-\nu}),\quad \lambda_{\nu}=\frac{\alpha_{\nu}}{M_{\nu}},
\end{equation*}
\begin{equation*}
\eps_{\nu+1}=\frac{c_{1}\eps_{\nu}^{\kappa}}{(\alpha_{\nu}\s^{t}_{\nu})^{\kappa-1}},\quad \s_{\nu+1}=\frac{\s_{\nu}}2,\quad \eta_{\nu}^{3}=\frac{\eps_{\nu}}{\alpha_{\nu}\s^{t}_{\nu}},
\end{equation*}
and 
\begin{equation*}
s_{\nu+1}=s_{\nu}-5\s_{\nu},\quad r_{\nu+1}=\eta_{\nu}r_{\nu}.
\end{equation*}
The initial conditions are chosen in the following way : $\s_{0}=s_{0}/40\leq 1/4$ so that $s_{0}>s_{1}>\cdots \geq s_{0}/2$,
\begin{equation*}
\eps_{0}= \gamma_{0}\alpha_{0}\s_{0}^{t}\quad \text{and}\quad \gamma_{0}= \(c_{0}+2^{t+3}c_{1}\)^{-3},
\end{equation*}
where $c_{0}$ is the constant which appears in \eqref{Init}. We also define $K_{\nu}=K_{0}2^{\nu}$ with $K^{\tau+1}_{0}=1/(c_{1}\gamma_{0})$.\\[3pt]
With the notation $D_{\nu}=D(s_{\nu},r_{\nu})$ we have 
\begin{lemm}[Iterative lemma, \cite{Poschel}]\label{Lem.Iterative}
 Suppose that $H_{\nu}=N_{\nu}+P_{\nu}$ is given on $D_{\nu}\times \Pi_{\nu}$, where $N_{\nu}=\om_{\nu}(\xi)\cdot y+\Omega_{\nu}(\xi)\cdot z\z$ is a normal form satisfying $\dis |\om_{\nu}|_{\Pi_{\nu}}^{\L}+|\Omega_{\nu}|^{\L}_{2\beta,\Pi_{\nu}}\leq M_{\nu}$,
\begin{equation*}
 |k\cdot \omega_{\nu}(\xi)+l\cdot \Omega_{\nu}(\xi)|\geq \alpha_{\nu}\frac{\<l\>}{A_{k}}, \quad (k,l)\in \mathcal{Z},
\end{equation*}
on $\Pi_{\nu}$ and 
$$\<P\>_{r_{\nu},D_{\nu}}^{\lambda_{\nu}}+\|X_{P}\|^{\lambda_{\nu}}_{r_{\nu},D_{\nu}}\leq \eps_{\nu}.$$
Then there exists a Lipschitz family of real analytic symplectic coordinate transformations $\Phi_{\nu+1}: D_{\nu+1}\times \Pi_{\nu} \longrightarrow D_{\nu}$ and a closed subset
\begin{equation*}
\Pi_{\nu+1}=\Pi_{\nu}\backslash \bigcup_{|k|>K_{\nu}}\mathcal{R}_{kl}^{\nu+1}(\alpha_{\nu+1}),
\end{equation*}
of $\Pi_{\nu}$, where 
\begin{equation*}
\mathcal{R}_{kl}^{\nu+1}(\alpha_{\nu+1})=\Big\{ \xi\in \Pi_{\nu}\,:\,|k\cdot \omega_{\nu+1}+l\cdot \Omega_{\nu+1}|<\alpha_{\nu+1}\frac{\<l\>}{A_{k}}\Big\}, 
\end{equation*}
such that for $H_{\nu+1}=H_{\nu}\circ \Phi_{\nu+1}=N_{\nu+1}+P_{\nu+1}$, the same assumptions are satisfied with $\nu+1$ in place of $\nu$.
\end{lemm}

We don't give the details of the proof of this result, since it is entirely done in \cite{Poschel} : it is of course an induction on $\nu\in \N$ which essentially relies on the results of the Section \ref{Sect.KAM}.\\

\subsection{Proof of Theorem \ref{thmKAM}}\label{Proof}~\\[5pt]

The result of Theorem \ref{thmKAM} is the convergence of the sequence $H_{\nu}$ to a Hamiltonian in normal form, for parameters $\xi$ in a set $\Pi_{\alpha}$, which is the limit of the sets $\Pi_{\nu}$.\\
We again follow the proof of  P\"oschel and we recall the following Lemma 
\begin{lemm}[Estimates, \cite{Poschel}]\label{Lem.est}
For $\nu\geq 0$, 
\begin{equation*}
\frac{1}{\s_{\nu}}\|\Phi_{\nu+1}-id\|^{\lambda_{\nu}}_{r_{\nu},D_{\nu+1}},\|D\Phi_{\nu+1}-I\|^{\lambda_{\nu}}_{r_{\nu},r_{\nu},D_{\nu+1}}\leq \frac{C\eps_{\nu}}{\alpha_{\nu}\s^{t}_{\nu}},
\end{equation*}
\begin{equation*}
|\om_{\nu+1}-\om_{\nu}|^{\lambda_{\nu}}_{\Pi_{\nu}},|\Omega_{\nu+1}-\Omega_{\nu}|^{\lambda_{\nu}}_{2\b,\Pi_{\nu}}\leq C\eps_{\nu}.
\end{equation*}
\end{lemm} 
\noindent Set $\Pi_{0}=\Pi\backslash \bigcup_{k,l}\mathcal{R}_{kl}^{\alpha_{0}}$ and $\Pi_{\alpha}=\cap_{\nu\geq 1} \Pi_{\nu}$. 
The proof that  $\text{Meas}(\Pi\backslash \Pi_{\alpha})\longrightarrow 0$ when $\alpha\longrightarrow 0$   is done in \cite[Section 5]{Poschel} and we do not repeat it here.\\
For $\nu\geq 1$ we define the map 
$$\Phi^{\nu}=\Phi_{1}\circ \cdots \circ \Phi_{\nu}\,:D_{\nu}\times \Pi_{\nu-1}\longrightarrow D_{\nu-1},$$
and thus we have $H_{\nu}=H\circ \Phi^{\nu}$. With the Lemma \ref{Lem.est} and since\footnote{here  we use the notation  $D(s/2)=D(s/2, 0)$. }  $\cap_{\nu\geq 1}D_{\nu}\times \Pi_{\nu}=D(s/2)\times \Pi_{\alpha}$, we are then able to show, as in \cite{Poschel}, that $\Phi^{\nu}$ is a Cauchy sequence for the supremum norm on  ${D(s/2)}\times \Pi_{\alpha}$. Thus it converges uniformly on  ${D(s/2)}\times \Pi_{\alpha}$ and its limit $\Phi$
 is  real analytic on $D(s/2)$. Further,  the estimate \eqref{Est.Phi} holds on $D(s/2)\times \Pi_{\alpha}$.\\
  It remains to prove that $\Phi$ is indeed defined on $D(s/2, r/2)\times \Pi_{\alpha}$ with the same estimate.  By Corollary \ref{coro_sol}  all the transforms $\Phi^{\nu}$ are linear in $y$ and quadratic in $z,\bar z$ and  thus the same is true for the transform $\Phi$ (this  fact was also used in\cite{Posch} or \cite{ElKuk1}).  This specific form is stable by composition and thus all the $\Phi^{\nu}$ have this form and in particular they are linear in $y$ and quadratic in $z,\bar z$.\\
 Therefore it suffices to verify that the first derivatives with respect to $y, z,\bar z$ and the second derivatives with respect to  $ z,\bar z$ of $\Phi^{\nu}$ are uniformly convergent on $D(s/2)\times \Pi_{\alpha}$ to conclude that $\Phi^{\nu}$ convergences    to $\Phi$ (actually an extension of the previously defined $\Phi$) uniformly on $D(s/2, \rho)\times \Pi_{\alpha}$ for any $\rho$. In particular, for $r$ small enough,
 $$\Phi\ : D(s/2, r/2)\times \Pi_{\alpha}\to D(s, r)$$
 and $\Phi$ still satisfies estimate \eqref{Est.Phi}.\\
 So it remains to analyse the convergence of the derivatives. Using Lemma  \ref{Lem.est} we obtain successively
 $\|D\Phi_{\nu}\|_{r_{\nu},r_{\nu},D_{\nu}}\leq 2$ and then uniformly on ${D(s/2)}\times \Pi_{\alpha}$
 $$\|D\Phi^{\nu+1}-D\Phi^{\nu}\|_{r_{\nu},r_{\nu},D_{\nu}}\leq \|D\Phi_{\nu}\|_{r_{\nu},r_{\nu},D_{\nu}} \|D\Phi_{\nu}-I|\|_{r_{\nu},r_{\nu},D_{\nu}}$$
 and we deduce that uniformly on ${D(s/2)}\times \Pi_{\alpha}$
 $$\|D\Phi^{\nu+1}-D\Phi^{\nu}\|_{r_{\nu},r_{\nu},D_{\nu}}\leq c\epsilon^{1/2}_\nu.$$
 So again $D\Phi^{\nu}$ converges uniformly on ${D(s/2)}\times \Pi_{\alpha}$. Similarly we obtain the convergence of the second derivatives using the  formula
 $$D^2 \Phi^{\nu+1}= D^2\Phi^{\nu}\cdot (D\Phi^{\nu})^2+ \Phi^{\nu}\cdot D^2\Phi^{\nu}.$$
On the other hand, 
again using Lemma  \ref{Lem.est}, the frequencies functions $\om_{\nu}$ and $\Omega_\nu$ converge uniformly on $ \Pi_{\alpha}$ to Lipschitz functions   $\om^\star$ and $\Omega^\star$ satisfying \eqref{Est.Freq} and thus \eqref{NewNR} in view of Lemma \ref{Lem.Iterative}. \\
We then deduce that, uniformly on $D(s/2, r/2)\times \Pi_{\alpha}$, $$R_\nu:=H\circ \Phi^\nu -N_\nu \quad \longrightarrow \quad H\circ\Phi -N^\star=:R^\star$$ and since for all $\nu$ the Taylor expansion of $R_\nu$  contains only monomials $y^mz^q\bar z^{\bar q}$ with $2|m|+|q+\bar q|\geq 3$ the same property holds true for $R^\star$. \qed


\section{Application to the nonlinear Schr\"odinger equation }\label{Sect.5}


 Let $n\geq 1$ be an integer and $\nu, \eps>0$ be two small parameters so that $\nu\geq C_{0}\eps$, where $C_{0}>0$ is a  constant which will be defined later. Set  $\Pi=[-1,1]^{n}$. We consider a perturbation of the  one dimensional  Schr\"odinger equation with harmonic potential
\begin{equation}\label{nls}
i\partial_t u+\partial^{2}_{x} u -x^{2}u  -\nu  V(\xi,x)u= \eps |u|^{2m}u ,\quad
(t,x)\in\R\times {\R},
\end{equation}
where $m\geq 1$ is an integer and $\big(V(\xi,\cdot)\big)_{\xi\in \Pi}$ is family of  a real analytic bounded potentials with $V(0,\cdot)=0$ which will be made explicit below. \\
Recall  that  $T=-\partial^{2}_{x} +x^{2}$ denotes the harmonic oscillator. Its eigenfunctions are the Hermite functions $(h_{j})_{j\geq 1}$, associated to the eigenvalues $(2j-1)_{j\geq 1}$. Now consider  the linear operator $A=A(\nu,\xi)=-\partial_{x}^{2} +x^2+\nu V(\xi,x)$. Under the previous assumptions,  $A$ is self-adjoint and has pure point spectrum with simple eigenvalues $(\lambda_j(\nu,\xi))_{j\geq 1}$ satisfying $\lambda_j(\nu,\xi)\sim 2j-1$. Its  eigenfunctions  $\big(\phi_j(\xi,\cdot )\big)_{j\geq 1}$ form an orthonormal basis of $L^2(\R)$, and $\phi_{j}(\xi,\cdot) \sim h_{j}$ as $\nu \to 0$ in $L^2$ norm. As a consequence  $A$ and $T$ have the same domain and $D(A^{p/2})=\H^{p}$. We will prove these facts for the particular class of potentials we will consider (see Lemmas \ref{lem.61} and \ref{lem.asymp} below). \\[5pt]
The parameter $\eps>0$ will be small so that we can apply Theorem \ref{thmKAM} and  $\nu>0$ will be  small too, so that we have a suitable perturbation theory for the operator $A$. We now explain the restriction $\nu \geq C_{0}\eps$. The aim of this section is to construct a potential $V$ so that Theorem \ref{thmKAM} applies, and in particular, \eqref{dio} has to be satisfied. Small values of $k$, $l$ in \eqref{dio} and the asymptotics of Lemma \ref{lem.asymp} give $C\nu \geq \alpha$. This together with the condition $\eps \leq \eps_{0} \alpha$ in Theorem \ref{thmKAM} yields the result. \\[5pt] 
We fix a finite subset $\J$ of $\N$ of cardinal $n$. Without loss of generality and in order to simplify the presentation, we assume $\J=\{1,\cdots,n\}$.  We then expand $u$ and $\bar u$ in the  basis of eigenfunctions using the phase space structure of the introduction, namely we write
\begin{align*}
u(x)=&\sum_{j=1}^n(y_{j}+I_{j})^{\frac12}\e^{i\theta_{j}}\phi_{j}(\xi,x)+\sum_{j\geq 1}z_{j}\phi_{j+n}(\xi,x),\\
\bar u(x) =&\sum_{j=1}^n(y_{j}+I_{j})^{\frac12}\e^{-i\theta_{j}}\phi_{j}(\xi,x)+\sum_{j\geq 1}\z_{j}\phi_{j+n}(\xi,x),
\end{align*}
where $(\theta,y,z,\bar z)\in  \mathcal{P}^{p}=\T^{n}\times \R^{n}\times {\ell}^{2}_{p} \times {\ell}^{2}_{p}$ (recall that $\ell^{2}_{p}$ is the space $\ell^{2}_{\Psi}$ with $\Psi(j)=j^{p/2}$) are regarded as variables and $I\in \R_+^n$ are regarded as parameters (here $\R_+$ denotes the set of non negative real numbers). In this setting equation \eqref{nls} reads as the Hamilton equations associated to the Hamiltonian function $H=N+P$ where
$$N=\sum_{j=1}^n\lambda_j (\nu,\xi) y_j+ \sum_{j\geq 1}\Lambda_{j}(\nu,\xi) z_j\bar z_j,$$
 $\Lambda_{j}(\nu,\xi)= \lambda_{j+n}(\nu,\xi)$, $G(u,\bar u)=(u\,\bar u)^{m+1}$ and
 \begin{align}\begin{split}\label{PNLS}
  P(\theta,y,z,\z)=\eps \int_{\R}G\Big(&\sum_{j=1}^{n}(y_{j}+I_{j})^{\frac12}\e^{i\theta_{j}}\phi_{j}(\xi,x)+\sum_{j\geq 1}z_{j}\phi_{j+n}(\xi,x),\\ &\sum_{j=1}^{n}(y_{j}+I_{j})^{\frac12}\e^{-i\theta_{j}}\phi_{j}(\xi,x)+\sum_{j\geq 1}\z_{j}\phi_{j+n}(\xi, x)\Big)\text{d}x.
\end{split}  \end{align}
For the sequel we fix $(I_{j})_{1\leq j\leq n}$. We assume that $(\theta,y,z,\bar z)\in D(s,r)$ for some fixed $s,r>0$ (recall the   definition \eqref{def.dsr}   of $D(s,r)$). There is no particular  smallness assumption on $s,r$, we only have to take $r>0$ with  $r<\min_{1\leq j\leq n}I_{j}$ so that $(y_{j}+I_{j})^{1/2}$ is well-defined.\\

We now show that we can construct a class of potentials $V$ so that Theorem \ref{thmKAM} applies.\\
\subsection{Definition of the family of potentials $V$}~\\[5pt]
Let $(f_{j})_{1\leq j\leq n}$ be the dual basis of $(h^{2}_{j})_{1\leq j\leq n}$, i.e. $(f_{j})\in \text{Span}_{\R}(h_{1}^{2},\dots,h_{n}^{2})$ and $\int_{\R}f_{j}h_{k}^{2}=\delta_{jk}$ for all $1\leq j,k\leq n$.\\
We say that $\a=(\a_{k})_{k\geq n+1} \in \mathcal{Z}_{n}$ if $-\frac12\leq \a_{k}\leq \frac12$ for all $k\geq n+1$. We endow the set of such sequences by the probability measure defined as the infinite product $(k\geq n+1)$ of the Lebesgue measure on $[-1/2,1/2]$. Then define 
\begin{equation*} 
g(x)=\sum_{k\geq n+1}\a_{k}\e^{-k}h_{2k-1}(\sqrt{2}x),
\end{equation*} 
and for $\xi=(\xi_{1},\dots , \xi_{n})\in\Pi=[-1,1]^{n}$ and
\begin{equation}\label{def.V}
V(\xi,x)=\sum_{k=1}^{n}\xi_{k}f_{k}(x)+\xi_{1}g(x).
\end{equation}
The spectral data $\phi_{j}$ and $\lambda_{j}$ are defined by the spectral  equation
\begin{equation}\label{eq.vp}
\big(-\partial^{2}_{x}+x^{2}+\nu V(\xi,x)\big)\phi_{j}(\xi,x)=\lambda_{j}(\nu\xi)\phi_{j}(\xi,x),
\end{equation}
and we assume that  the $(\phi_{j})$ are $L^{2}-$normalised ($\|\phi_{j}(\xi,\cdot)\|_{L^{2}}=1$ for all $\xi\in \Pi$ and $j\geq 1$). Moreover, in order to define $\phi_{j}$ uniquely, we impose $\<\phi_{j},h_{j}\>>0.$\\[4pt]
In the sequel we need a particular case of estimates proved by K. Yajima \& G. Zhang \cite{YajimaZhang1}
\begin{lemm}[\cite{YajimaZhang1}]
For all $2<p<\infty$ there exists $\alpha>0$ and $C>0$ so that for all $\xi \in \Pi$ and $j\geq 1$ 
\begin{equation}\label{YZ}
\|\phi_{j}(\xi,\cdot)\|_{L^{p}(\R)}\leq Cj^{-\a}.
\end{equation}
 \end{lemm}
 The next result is the key estimate in our perturbation theory.
 \begin{lemm}\label{lem.61} There exist $\a>0$ and $C>0$ so that for all $\xi \in \Pi$, $\nu>0$ and $j\geq 1$ 
\begin{equation}\label{norme.2}
\|\phi_{j}(\xi,\cdot)-\phi_{j}(\eta,\cdot)\|_{L^{2}}\leq C\nu |\xi-\eta|j^{-\a}.
\end{equation}
 \end{lemm}
In particular $\|\phi_{j}(\xi,\cdot)-h_{j}\|_{L^{2}}\leq C\nu |\xi|j^{-\a}$, which shows that the $\phi_{j}$ are close to the Hermite functions in $L^2$ norm.  
\begin{proof} In the sequel, we write $\phi_{j}(\xi)$ instead of $\phi_{j}(\xi,\cdot)$.
For $\xi,\eta \in \Pi$, we compute 
\begin{equation*}
A(\nu\,\xi)\phi_{j}(\eta)=\big(-\partial^{2}_{x}+x^{2}+\nu V(\xi,x)\big)\phi_{j}(\eta)=\lambda_{j}(\nu\,\eta)\phi_{j}(\eta)+\nu (V(\xi,x)-V(\eta,x))\phi_{j}(\eta).
\end{equation*}
Thus by \eqref{def.V} and \eqref{YZ} there exists $\a>0$ such that 
\begin{eqnarray}\label{CV.0}
\big\|\big(A(\nu\,\xi)-\lambda_{j}(\nu\,\eta)\big)\phi_{j}(\eta)\big\|_{L^{2}}&=&\nu \|(V(\xi)-V(\eta))\phi_{j}(\eta)\|_{L^{2}}\nonumber\\
&\leq &\nu\|V(\xi)-V(\eta)\|_{L^{4}}\|\phi_{j}(\eta)\|_{L^{4}}\nonumber \\
&\leq& C\nu |\xi-\eta|{j^{-\a}}.
\end{eqnarray}
Choosing  $\eta=0$ in \eqref{CV.0}, and as  $\phi_{j}(0)=h_{j}$ and $\lambda_{j}(0)=2j-1 $,  we get  
\begin{eqnarray*}
1=\|h_{j}\|_{L^{2}} &\leq& \big\|\big(A(\nu\,\xi)-(2j-1)\big)^{-1}\big\|_{L^{2}\to L^{2}}\big\|\big(A(\nu\,\xi)-(2j-1)\big)h_{j}\big\|_{L^{2}} \\
&\leq & C\nu{j^{-\a}}\big\|\big(A(\nu\,\xi)-(2j-1)\big)^{-1}\big\|_{L^{2}\to L^{2}}.
\end{eqnarray*}
The previous estimate together with the general formula which holds for any self-adjoint operator $ \|\big(A(\nu\,\xi)-(2j-1)\big)^{-1}\|_{L^{2}\to L^{2}} = \text{dist}\big(2j-1,\s(A(\nu\,\xi))\big)^{-1}$ gives $\text{dist}\big(2j-1,\s(A(\nu\,\xi))\big) \leq  C\nu{j^{-\a}}$, where $\s(A(\nu\,\xi))$ denotes the spectrum of $A(\nu\,\xi)$.  A similar argument, taking $\xi=0$ in \eqref{CV.0}, leads to  $\text{dist}\big(\lambda_{j}(\nu \eta),\s(T)\big) \leq  C\nu{j^{-\a}}$. Thus for all $j\geq 1$ 
\begin{equation}\label{dvp.vp}
\lambda_{j}(\nu \xi)=2j-1+\nu \ O( j^{-\a}).
\end{equation}
Using that $(\phi_{k}(\xi))_{k\geq 1}$ is a Hilbertian basis of $L^{2}(\R)$, we deduce 
\begin{eqnarray}
\big\|\phi_{j}(\eta)-\<\phi_{j}(\xi),\phi_{j}(\eta)\>\phi_{j}(\xi)\big\|^{2}_{L^{2}}&=& \sum_{k\geq 1} |\<\phi_{j}(\eta)-\<\phi_{j}(\xi),\phi_{j}(\eta)\>\phi_{j}(\xi),\phi_{k}(\xi)\>|^{2}\nonumber\\
&=& \sum_{k\geq 1, k\neq j} |\<\phi_{j}(\eta),\phi_{k}(\xi)\>|^{2}.\label{c.1}
\end{eqnarray}
With the same decomposition, we can also write
\begin{eqnarray}
\|\big(A(\nu\,\xi)-\lambda_{j}(\nu\,\eta)\big)\phi_{j}(\eta)\|^{2}_{L^{2}}&=& \sum_{k\geq 1} |\< \big(A(\nu\,\xi)-\lambda_{j}(\nu\,\eta)\big)\phi_{j}(\eta),\phi_{k}(\xi) \>|^{2}\nonumber \\
&=&\sum_{k\geq 1} |\< \big(\lambda_{k}(\nu\,\xi)-  \lambda_{j}(\nu\,\eta) \big)\phi_{k}(\xi),\phi_{j}(\eta) \>|^{2} \nonumber\\
&=&\sum_{k\geq 1}|\lambda_{k}(\nu\,\xi)-  \lambda_{j}(\nu\,\eta) |^{2}|\<\phi_{k}(\xi),\phi_{j}(\eta)\>|^{2}\nonumber \\
&\geq & \sum_{k\geq 1,k\neq j}|\<\phi_{k}(\xi),\phi_{j}(\eta)\>|^{2} \label{c.2},
\end{eqnarray}
because by \eqref{dvp.vp} $|   \lambda_{k}(\nu\,\xi)-  \lambda_{j}(\nu\,\eta)   |\geq 1$ for $k\neq j$ uniformly in $\xi,\eta$ and uniformly in $\nu$ small enough. Now by \eqref{CV.0}, \eqref{c.1} and \eqref{c.2} we deduce that 
$$\big\|\phi_{j}(\eta)-\<\phi_{j}(\xi),\phi_{j}(\eta)\>\phi_{j}(\xi)\big\|^{2}_{L^{2}} \leq C\nu |\xi-\eta| j^{-\a} .$$
 In particular, taking the scalar product of $\phi_j(\eta)$ with $\phi_{j}(\eta)-\<\phi_{j}(\xi),\phi_{j}(\eta)\>\phi_{j}(\xi)$, we obtain 
 $$\big| 1-\<\phi_{j}(\xi),\phi_{j}(\eta)\>^2\big|\leq C\nu |\xi-\eta| j^{-\a} .
 $$
  The last two estimates imply $\|\phi_{j}(\xi)-\phi_{j}(\eta)\|_{L^{2}}\leq  C\nu |\xi-\eta| j^{-\a} $ which was the claim.
\end{proof}
\begin{lemm}\label{lem.asymp}
We have the following asymptotics when $\nu\longrightarrow 0$ 
\begin{equation}\label{dvp1}
\lambda_{j}(\nu\,\xi)=2j-1+\nu \xi_{j}+o(\nu),\quad \forall \,   1\leq j\leq n,
\end{equation}
\begin{equation}\label{dvp2}
\Lambda_{j}(\nu\,\xi)=\lambda_{j+n}(\nu \xi)=2(j+n)-1+\nu\sum_{k=1}^{n}\xi_{k} \int_{\R} (f_{k}+\delta_{1k}g)h^{2}_{n+j}+o(\nu),\quad \forall \,   j\geq 1.
\end{equation}
 \end{lemm}
\begin{proof}
We first prove \eqref{dvp1}. 
 We differentiate equation \eqref{eq.vp} in $\xi_{k}$
\begin{equation*}
A(\nu\,\xi)\frac{\phi_{j}(\xi)}{\partial \xi_{k}}+\nu (f_{k}+\delta_{1k}g)\phi_{j}(\xi)=\lambda_{j}(\nu\,\xi)\frac{\phi_{j}(\xi)}{\partial \xi_{k}}+\frac{\partial \lambda_{j}(\nu\,\xi)}{\partial \xi_{k}}\phi_{j}(\xi),
\end{equation*} 
 take the scalar product with $\phi_{j}(\xi)$ and  the selfadjointness of $A(\nu\,\xi)$ gives 
 \begin{equation}\label{derivee}
\frac{\partial \lambda_{j}(\nu\,\xi)}{\partial \xi_{k}}=\nu \int_{\R}(f_{k}+\delta_{1k}g)\phi_{j}^{2}(\xi).
\end{equation}
Now by \eqref{norme.2}
\begin{eqnarray*}
|\int_{\R}(f_{k}+\delta_{1k}g)(\phi_{j}^{2}(\xi)-h^{2}_{j})|&\leq &\|f_{k}+\delta_{1k}g\|_{L^{\infty}}\|\phi_{j}(\xi)+h_{j}\|_{L^{2}}\|\phi_{j}(\xi)-h_{j}\|_{L^{2}}\nonumber \\
&\leq & C \|\phi_{j}(\xi)-h_{j}\|_{L^{2}}\longrightarrow 0
\end{eqnarray*} 
when $\nu\longrightarrow 0$. Thus by   definition of the $f_{k}$ and $g$ and by estimate \eqref{derivee}, we obtain that for all $1\leq j\leq n$
\begin{eqnarray*}
\lambda_{j}(\nu\,\xi)&=&2j-1+\nu \sum_{k=1}^{n}\xi_{k}\int_{\R}(f_{k}+\delta_{1k}g)h_{j}^{2}+o(\nu)\\
&=& 2j-1 +\nu \xi_{j}+ o(\nu),
\end{eqnarray*}
which is \eqref{dvp1}.\\
The asymptotic of \eqref{dvp2} is proved in the same way. Observe that we can prove a better estimate on the error term using \eqref{dvp.vp}, but we do not need it here.
 \end{proof}
 ~
\subsection{Verification of Assumptions \ref{AS1} and \ref{AS2}}~\\[5pt]
  \begin{lemm}\label{Lem.Z}
There exists a null measure set $\mathcal{N}\subset \mathcal{Z}_{n}$ such that for all $\a \in \mathcal{Z}_{n} \backslash \mathcal{N}$ we have for all $1\leq p,q$, with $p\neq q$
\begin{equation}\label{proj1}
\int_{\R}(f_{1}+g)h^{2}_{n+p}\notin \Z,
\end{equation}
and 
\begin{equation}\label{proj2}
 \int_{\R}(f_{1}+g)(h^{2}_{n+p}\pm h^{2}_{n+q})\notin \Z.
\end{equation}
 \end{lemm}
\begin{proof}
For $j\geq 1$, the Hermite function $h_{j}$ reads $h_{j}(x)=P_{j}(x)\e^{-x^{2}/2}$, where $P_{j}$ is a polynomial of degree exactly $(j-1)$, and $P_{j}$ is even (resp. odd) when $(j-1)$ is even (resp. odd). We have $\text{Span}_{\R}(h_{1},\dots,h_{n})=\e^{-x^{2}/2}\R_{n-1}[X]$. Thus we deduce that there exist $(\mu_{kj})$ so that 
\begin{equation}\label{hj2}
h^{2}_{j}(x)=\sum_{k=1}^{j}\mu_{kj}h_{2k-1}(\sqrt{2}x),
\end{equation}
with $\mu_{jj}\neq 0$.\\
We assume that $q<p$. The application 
\begin{equation*}
(\a_{n},\a_{n+1},\dots)\longmapsto  \int_{\R}(f_{1}+g)(h^{2}_{n+p}\pm h^{2}_{n+q})
\end{equation*}
is a linear form. In order to prove \eqref{proj2}, it suffices to check that this linear form is nontrivial. According to \eqref{hj2} and to the definition of $f_{1}$ and $g$, the coefficient of $\a_{n+p}$ is 
$$\e^{-(n+p)}\mu_{n+p,n+p}\int_{\R}h^{2}_{2(n+p)-1}(\sqrt{2}x)\text{d}x=\e^{-(n+p)}\mu_{n+p,n+p}/\sqrt{2}\neq 0.$$ Therefore for fixed $p,q$ \eqref{proj2} is satisfied on the complementary of a null measure set $\mathcal{N}_{p,q}$. Finally, \eqref{proj2} is satisfied on $\mathcal{Z}_{n} \backslash \mathcal{N}$ where $ \mathcal{N}=\cup_{p,q\geq 1} \mathcal{N}_{p,q}$. The proof of \eqref{proj1} is similar.
  \end{proof}

  In the sequel we fix $\a \in \mathcal{Z}_{n} \backslash \mathcal{N}$ so that Lemma \ref{Lem.Z} holds true.
  We are now able to show that Assumption \ref{AS1} is satisfied. Recall that in our setting, the internal frequencies are $\lambda(\nu\xi)=(\lambda_{j}(\nu\xi))_{1\leq j\leq n}$ and the external frequencies are $\Lambda(\nu\xi)=(\Lambda_{j}(\nu\xi))_{j\geq 1}$ with $\Lambda_{j}(\nu\xi)=\lambda_{n+j}(\nu\xi)$. 
\begin{lemm}
There exists $\nu_{0}>0$ so that for all $0<\nu <\nu_{0}$ we have 
\begin{equation}\label{mesnulle}
  \text{Meas}\Big(\big\{\,\xi \in \Pi\;:\:k\cdot \lambda(\nu\,\xi)+l\cdot \Lambda(\nu\,\xi)=0\,\big\}\Big) =0,\quad \forall\; (k,l)\in \mathcal{Z},
  \end{equation}
  and  for all $\xi \in \Pi$
\begin{equation}\label{non.nul}
l\cdot \Lambda(\nu\,\xi)\neq 0, \quad \forall\, 1\leq |l|\leq 2.
\end{equation}
 \end{lemm}
\begin{proof}
We prove \eqref{mesnulle} by contradiction. Let $(k,l)\in \mathcal{Z}$. In the case  $|l|=2$ in \eqref{mesnulle} we can write 
\begin{equation*}
k\cdot \lambda(\nu\,\xi)+l\cdot \Lambda(\nu\,\xi)=\sum_{j=1}^{n}k_{j}\lambda_{j}(\nu\,\xi)+\lambda_{n+p}(\nu\,\xi)-\lambda_{n+q}(\nu\,\xi):=F(\nu\, \xi),
\end{equation*}
for some $p,q\geq 1$. Now if  \eqref{mesnulle} does not hold, $F: \R^{n}\longrightarrow \R$ is a $\mathcal{C}^{1}$ function which vanishes on a set of positive measure in any neighbourhood of $0$, thus $F(0)=0$ and for all $1\leq k\leq n$, $\frac{\partial F}{\partial \xi_{k}}(0)=0$. By Lemma \ref{lem.asymp} these conditions read
 \begin{eqnarray}
\sum_{j=1}^{n}(2j-1)k_{j}+2(p-q)&=&0\quad \mbox{ and}\nonumber \\
k_{j}+\int_{\R}(f_{j}+\delta_{ij}g)(h^{2}_{n+p}-h^{2}_{n+q})&=&0, \quad \forall \;1\leq j\leq n. \label{dl}
\end{eqnarray}
In particular for $j=1$, \eqref{dl}  is in contradiction with \eqref{proj2}.\\
The case $|l|=1$ is similar, using \eqref{proj1}.\\
It remains to prove \eqref{non.nul}. For all $j\geq 1$,  $\Lambda_{j}(\nu\,\xi)\longrightarrow 2j-1$ when $\nu \longrightarrow0$. Hence \eqref{non.nul} holds true if $\nu$ is small enough.
  \end{proof}
  
We now check Assumption \ref{AS2}. Firstly, thanks to \eqref{dvp.vp} we have that for $j,k\geq 1$, $|\Lambda_{j}(\nu \xi)-\Lambda_{k}(\nu \xi)|\geq |j-k|$ and $|\Lambda_{j}(\nu \xi)|\geq j$. Then by \eqref{derivee} and \eqref{YZ}
\begin{eqnarray*}
|\Lambda_{j}(\nu \xi)-\Lambda_{j}(\nu \eta)| &\leq & \nu |\xi-\eta|\,\sup_{\xi \in\Pi}\int_{\R}\big| (f_{k}+\delta_{1k}g)\phi^{2}_{j+n}(\xi,\cdot) \big|\\
 &\leq & \nu |\xi-\eta| \big\|f_{k}+\delta_{1k}g\big\|_{L^{2}} \sup_{\xi \in\Pi}\|\phi_{j+n}(\xi,\cdot)\|^{2}_{L^{4}}\\
 &\leq &C \nu |\xi-\eta|j^{-\a},
\end{eqnarray*}
and Assumption \ref{AS2} is fulfilled.\\
\subsection{Verification of Assumptions \ref{AS3} and \ref{AS4}}~\\[5pt]
Recall that for $p\geq 0$,  $\H^{p}=D(T^{p/2})$ is  the Sobolev space based on the harmonic oscillator.  Thanks to  \eqref{norme.2} and \eqref{dvp.vp}, we also have $\H^{p}=D(A^{p/2}(\nu\,\xi))$ for all $\nu>0$ small enough and $\xi\in \Pi$. Observe  that  $\H^{p}$ is an algebra and the Sobolev embeddings which hold for the usual Sobolev space $H^{p}$ are also true here, since $H^{p}\subset \H^{p}$. \\
Let $u=\sum_{j\geq 1} \a_{j}\phi_{j}$. Then  $u\in \H^{p}$ if and only if  $\a_{j}\in  \ell^{2}_{p}$.\\[4pt]
We now check the smoothness of $P$ and the decay of the vector field $X_{P}$.\\
Let $p\geq 2$ so that we are in the framework of Theorem \ref{thmKAM}. Since  $G(u,\ov{u})=(u\ov{u})^{m+1}$ in \eqref{PNLS}, we have 
\begin{equation}\label{def.PI}
P=\eps \int_{\R}|u|^{2(m+1)}.
\end{equation}
We first  show that $\frac{\partial P}{\partial z_{j}} \in \ell^{2}_{p}$. We have 
\begin{equation}\label{na.z}
\frac{\partial P}{\partial z_{j}}=\eps (m+1)\int_{\R}\phi_{j+n}u^{m}\,\ov{u}^{m+1},
\end{equation}
thus $\frac{\partial P}{\partial z_{j}}$ is (up to a constant factor) the $(j+n)$th coefficient of the decomposition of $u^{m}\,\ov{u}^{m+1}$, and this latter term is in $\H^{p}$ (because $\H^{p}$ is an algebra), hence the result. The other components of $X_{P}$ can be handled in the same way, and we get $X_{P}\in \mathcal{P}^{p}$.\\
By \eqref{def.PI} and Sobolev embeddings
 \begin{equation}
\sup_{D(s,r)\times \Pi}|P|\leq \eps \|u\|_{L^{2(m+1)}}^{2(m+1)}\leq \eps \|u\|_{\H^{p}}^{2(m+1)}.
\end{equation}
Similarly, using \eqref{na.z} and 
 \begin{eqnarray*}
\frac{\partial P}{\partial \theta_{j}}&=&\eps i(m+1)(y_{j}+I_{j})^{\frac12}\Big[ \e^{i\theta_{j}} \int_{\R}\phi_{j}u^{m}\,\ov{u}^{m+1}+  \e^{-i\theta_{j}} \int_{\R}\phi_{j}u^{m+1}\,\ov{u}^{m}  \Big]\\
\frac{\partial P}{\partial y_{j}}&=&\eps \frac {m+1}2 (y_{j}+I_{j})^{-\frac12} \Big[\e^{i\theta_{j}}\int_{\R}\phi_{j}u^{m}\,\ov{u}^{m+1}+\e^{-i\theta_{j}}\int_{\R}\phi_{j}u^{m+1}\,\ov{u}^{m} \Big]
\end{eqnarray*}
it is easy to see that $\sup_{D(s,r)\times \Pi}| X_{P}|_{r}\leq C\eps$. We now turn to the Lipschitz norms. Let $\xi,\eta \in \Pi$
\begin{eqnarray}
|P(\xi)-P(\eta)|&\leq &C \eps \|u(\xi)-u(\eta)\|_{L^{2}}(\|u(\xi)\|^{2m+1}_{L^{4(2m+1)}}+\|u(\eta)\|^{2m+1}_{L^{4(2m+1)}})\nonumber \\
&\leq &C\eps  \|u(\xi)-u(\eta)\|_{L^{2}}  \|u\|_{\H^{p}}^{2m+1}.\label{a1}
\end{eqnarray}
Now by \eqref{norme.2}
\begin{eqnarray}
 \|u(\xi)-u(\eta)\|_{L^{2}}  &\leq & C\sum_{j=1}^{n}\| \phi_{j}(\xi)- \phi_{j}(\eta)\|_{L^{2}}+\sum_{j\geq 1} j^{p}|z_{j}|\| \phi_{j+n}(\xi)- \phi_{j+n}(\eta)\|_{L^{2}}\nonumber \\
  &\leq & C|\xi-\eta|,\label{a2}
\end{eqnarray}
where in the last line we used Cauchy-Schwarz and the fact that $(z_{j})_{j\geq 1}\in l^{2}_{p}$ with $p\geq 2$. Then \eqref{a1} and \eqref{a2} show the Lipschitz regularity of $P$. We can proceed similarly for $X_{P}$.\\
It remains to prove the decay estimates of Assumption \ref{AS4}. Using  \eqref{na.z}, \eqref{YZ} and the Sobolev embeddings, we obtain 
 \begin{equation*}
\Big|\frac{\partial P}{\partial z_{j}} \Big|\leq\eps  (m+1)\|\phi_{j+n}\|_{L^{\infty}(\R)}\|u\|^{2m+1}_{L^{2m+1}}\leq C\eps j^{-\a} \|u\|^{2m+1}_{\H^{p}},
\end{equation*}
and similarly, from 
 \begin{equation*}
\frac{\partial^{2} P}{\partial z_{j}\partial z_{l}}=\eps m(m+1)\int_{\R}\phi_{j+n}\phi_{l+n}u^{m-1}\,\ov{u}^{m+1},
\end{equation*}
we deduce 
 \begin{equation*}
\Big|\frac{\partial^{2} P}{\partial z_{j}\partial z_{l}} \Big|\leq C\eps \|\phi_{j+n}\|_{L^{\infty}(\R)}\|\phi_{l+n}\|_{L^{\infty}(\R)}\|u\|^{2m}_{L^{2m}}\leq \eps C(jl)^{-\a} \|u\|^{2m}_{\H^{p}}.
\end{equation*}
The estimates of the Lipschitz norms are obtained as in \eqref{a1}, \eqref{a2} and using \eqref{norme.2}.
 
As a conclusion Assumptions \ref{AS1} - \ref{AS4} are satisfied and  we can apply Theorem \ref{thmKAM} with some $\b>0$ if $\eps >0$ is small enough. Recall that $\Pi=[-1,1]^n$. 

\begin{theo}\label{main1}
Let $m\geq 1$ and $n\geq 1$ be two integers. Let $V(\xi, \cdot)$ be the $n$ parameters family of potentials defined by \eqref{def.V}. There exist $\eps_0>0$, $\nu_0>0$, $C_{0}>0$ and,  for each $\eps<\eps_0$, a  Cantor set $\Pi_{\eps}\subset \Pi$ of asymptotic  full measure when $\eps \to 0$, such that  for each $\xi\in \Pi_{\eps}$  and for each $C_{0}\eps\leq \nu<\nu_0$, the solution of 
\begin{equation}\label{nls2}
i\partial_t u+\partial^{2}_{x} u -x^{2}u  -\nu  V(\xi,x)u= \eps |u|^{2m}u ,\quad
(t,x)\in\R\times {\R}
\end{equation}
with initial datum
\begin{equation}\label{CI}
u_0(x)=\sum_{j=1}^n I_j^{1/2}e^{i\theta_j}\phi_j(\xi,x),
\end{equation}
with  $(I_1,\cdots,I_n)\subset (0,1]^n$ and $\theta\in \T^n$,
is quasi periodic with a quasi period $\omega^*$ close to $\omega_0
=(2j-1)_{j=1}^n$: $|\omega^*-\omega_0|<C\nu$.\\
More precisely, when $\theta$ covers $\T^n$, the set of solutions of \eqref{nls2} with initial datum \eqref{CI} covers a $n$ dimensional torus which is invariant by \eqref{nls2}. Furthermore this torus is linearly stable.
\end{theo}
 
\begin{rema}
From the proof it is clear that our result also applies to any non linearity which is a linear combination of $|u|^{2m}u$. Moreover, under ad hoc conditions on the derivatives of $G$, we can admit some  non linearities of the form $\frac{\partial G}{\partial {\ov{u}}}(x,u,\ov{u})$ (i.e. depending on $x$) in \eqref{nls}. Also we can replace the set $\{1,\cdots,n\}$ by any finite set of $\N$ of cardinality $n$.
\end{rema}

\section{Application to the linear Schr\"odinger equation} \label{Sect.6}

 
In this section we prove Theorem \ref{theo:LS} following the scheme developed  by H. Eliasson and S. Kuksin  in \cite{ElKuk2} for the linear Schr\"odinger equation on the torus with quasi-periodic in time potentials.\\
The setting differs slightly from Section \ref{Sect.5} since now we are not considering a perturbation around a finite dimensional torus but we want to construct a linear change of variable defined on all the phase space. Consider the equation
\begin{equation}\label{LS*}
 i\partial_t u=-\partial^{2}_{x} u +x^2u +\epsilon V(t\omega,x)u
 \end{equation}
where $V$ satisfies the condition \eqref{*}. Recall the definition of the phase space $ \mathcal{P}^{p}=\T^{n}\times \R^{n}\times {\ell}^{2}_{p} \times {\ell}^{2}_{p}$. Recall also that  $h_j$, $j\geq 1$ denote   the eigenfunctions of the quantum harmonic oscillator $T=-\partial^{2}_{x} +x^2$ and that   we have $Th_j= (2j-1) h_j$, $j\geq 1$. Expanding $u$ and $\bar u$ on the Hermite basis, $u=\sum_{j\geq 1} z_j h_j$, $\bar u=\sum_{j\geq 1} \bar z_j h_j$, equation \eqref{LS*} reads as a non autonomous Hamiltonian system
\begin{equation}\label{ls1}
\left\{ \begin{array}{ll}     
\dot z_j= -i(2j-1) z_j- i\eps\frac{\partial}{\partial \bar z_j}\widetilde{Q}(t,z,\bar z), & j\geq 1\\[4pt]
\dot{\bar z}_j=i(2j-1)\bar z_j+ i\eps\frac{\partial}{\partial  z_j}\widetilde Q(t,z,\bar z), & j\geq 1
\end{array}\right.
\end{equation} 
where 
$$
\widetilde Q(t,z,\bar z)=\int_\R V(\omega t,x)\big(\sum_{j\geq 1} z_j h_j(x)\big)\big(\sum_{j\geq 1} \bar z_j h_j(x)\big)\text{d}x
$$
and\footnote{For the moment we work in $\ell^2_{2}\times \ell^2_{2}$, the largest phase space in which our abstract result applies. } $(z,\bar z)\in \ell^2_{2}\times \ell^2_{2}$.
We then re-interpret  \eqref{ls1} as an autonomous Hamiltonian system in an extended phase space $\mathcal{P}^{2}$
\begin{equation}\label{ls2}
\left\{ \begin{array}{ll}     
\dot z_j=-i(2j-1) z_j- i\eps\frac{\partial}{\partial \bar z_j}Q(\theta,z,\bar z) & j\geq 1 \\
\dot{\bar z}_j= i(2j-1)\bar z_j+ i\eps\frac{\partial}{\partial  z_j}Q(\theta,z,\bar z) & j\geq 1\\
\dot \theta_j= \omega_j & j=1,\cdots, n\\
\dot y_j= -\eps\frac{\partial}{\partial \theta_j}Q(\theta,z,\bar z) & j=1,\cdots, n
\end{array}\right.
\end{equation} 
where
$$
Q(\theta,z,\bar z)=\int_\R V(\theta,x)\big(\sum_{j\geq 1} z_j h_j(x)\big)\big(\sum_{j\geq 1} \bar z_j h_j(x)\big)\text{d}x
$$
is quadratic in $(z,\bar z)$. We notice that the first three equations of \eqref{ls2} are independent of $y$ and are equivalent to \eqref{ls1}. Furthermore \eqref{ls2} reads as the Hamiltonian equations associated with the Hamiltonian function $H=N+Q$ where
$$
N(\omega)= \sum_{j=1}^n \omega_j y_j +\sum_{j\geq 1} (2j-1) z_j\bar z_j.
$$
Here the external parameters are directly the frequencies $\dis \om=(\omega_j)_{1\leq j\leq n}\in [0,2\pi)^n=:\Pi$ and the normal frequencies $\Omega_j=2j-1$ are constant. \\[5pt]
\subsection{Statement of the results and proof}~\\
 \begin{theo}\label{thmKAM2}
 There exists $\eps_0>0$ such that if $0<\eps<\eps_{0}$, there exist 
 \begin{enumerate}[(i)]
 \item a Cantor set $\Pi_\eps\subset \Pi$ with $\text{Meas}(\Pi\backslash \Pi_\eps)\rightarrow 0$ as $\eps\rightarrow 0$ ;
\item a Lipschitz family of real analytic, symplectic and  linear coordinate transformation $\Phi: \Pi_\eps \times \mathcal P^{0} \rightarrow \mathcal P^{0}$ of the form 
\begin{equation} \label{specific}
\Phi_\om(y,\theta,Z)= (y+\frac 1 2 Z\cdot  M_\om(\theta)Z,\theta,L_\om(\theta)Z)
\end{equation}
where $Z=(z,\bar z)$,  $L_\om(\theta)$ and $M_\om(\theta)$ are linear bounded operators from $\ell^2_{p}\times\ell^2_{p} $ into itself for all $p\geq 0$ and $L_\om(\theta)$ is invertible ;
\item a Lipschitz family of new normal forms 
 \begin{equation*} 
 N^\star(\om)=\sum_{j=1}^n \omega_j  y_j+ \sum_{j\geq 1}\Omega_j^\star(\om) z_j\bar z_j \;;
 \end{equation*} 
 \end{enumerate}
 such that  on $\Pi_{\eps}\times \mathcal{P}^{0}$
 \begin{equation*}
 H\circ \Phi = N^\star.
 \end{equation*}
 Moreover 
 the new external frequencies  are close to the original ones
 \begin{equation*}
|\Omega^\star-\Omega|_{2\beta,\Pi_\eps}\leq c\eps ,
\end{equation*} 
 and the new frequencies satisfy a non resonant condition, there exists $\alpha >0$ such that for all $\om \in \Pi_{\eps}$
  \begin{equation*}
 \big| k\cdot \om +l \cdot \Omega^\star(\om) \big|\geq {\alpha} \  \frac{ \<l\>}{1+|k|^{\tau}},\quad (k,l)\in \mathcal{Z} .
 \end{equation*}
 \end{theo}
Notice that in the new coordinates, $(y',\theta',z',\bar z')=\Phi_{\om}^{-1}(y,\theta,z,\bar z)$, the dynamic is linear with  $y'$  invariant :
\begin{equation}\label{ls3}
\left\{ \begin{array}{ll}     
\dot z'_j=i\Omega^\star_j z'_j & j\geq 1 \\
\dot{\bar z}'_j= -i\Omega^\star_j\bar z'_j  & j\geq 1\\
\dot \theta'_j= \omega_j & j=1,\cdots, n\\
\dot y'_j= 0 & j=1,\cdots, n.
\end{array}\right.
\end{equation} 
  As \eqref{LS*} is equivalent to \eqref{ls2}, this theorem implies Theorem \ref{theo:LS}. In particular the  solutions $u(t,x)$ of \eqref{LS*} with initial datum $u_0(x)=\sum_{j\geq 1} z_{j}(0)h_j(x)$ read $u(t,x)=\sum_{j\geq 1} z_j(t)h_j(x)$ with
\begin{equation*}\label{fin}
 (z,\bar z)(t)=L_\om(\om t)(z'(0)e^{i\Omega^\star t},\bar z'(0)e^{-i\Omega^\star t})
 \end{equation*}
 and $(z'(0),\bar z'(0))= L_\om^{-1}(0)(z(0),\bar z(0))$.\\
 Thus 
 $$
 u(t,x)=\sum_{j\geq 1} \psi_j(\om t,x)e^{i\Omega^\star_j t}
 $$
 where $\psi_j (\theta,x)=\sum_{\ell\geq 1} [L_\om(\theta)L_\om^{-1}(0)(z(0),\bar z(0))]_\ell h_\ell(x)$.\\
In particular the solutions are all almost periodic in time with a non resonant frequencies vector $(\om,\Omega^\star)$.  
Furthermore we observe that $ \psi_j(\om t,x)e^{i\Omega^\star_j t}$ solves \eqref{LS*} if and only if $\Omega^\star_j +k\cdot \om$ is an eigenvalue of \eqref{Floq} (with eigenfunction  $ \psi_j(\theta,x)e^{i\theta\cdot k}$). This shows that the spectrum of the Floquet operator \eqref{Floq} equals $\{\Omega^\star_j +k\cdot \om \mid k\in\Z^n, \ j\geq 1 \}$ and thus Corollary 1.4 is proved.
\begin{rema}
Although $\Phi$ is defined on $ \mathcal P^{0} $, the normal forms $N$ and $N^\star$ are  well defined on $ \mathcal P^{p} $ only when $p\geq 1/2$. Nevertheless their flows are well defined and continuous from  $\mathcal P^{0}$ into itself (cf. \eqref{ls3}).
\end{rema}

 \begin{proof} 
Let $\tilde \Pi\subset \Pi$ be the subset of Diophantine vector of frequencies $\om$, i.e.   having the property that there exists $0< \alpha\leq 1$ such that
\begin{equation}\label{diophantine}
 \big|k\cdot \om -b\big|\geq     \frac{{2\pi\alpha}}{|k|^{\tau-1}},\quad k\in \Z^n\setminus\{0\}, \ b\in \Z 
 \end{equation}
 for some $\tau>n+2$. It is well known that $\text{Meas}(\Pi\backslash \tilde \Pi)=0$. Further this Diophantine condition implies  that 
  \begin{equation*}
 \big|k\cdot \om +l\cdot \Omega \big|\geq  {\alpha} \  \frac{ \<l\>}{1+|k|^{\tau}},\quad (k,l)\in \mathcal{Z},
 \end{equation*}
 since $l\cdot\Omega \in \Z$ and if $  \<l\>\leq 2\pi |k|$ then $\frac{2{\pi\alpha}}{|k|^{\tau-1}}\geq \alpha \frac{ \<l\>}{1+|k|^{\tau}}$ while if $ \<l\>\geq 2\pi|k|$ then $\big|k\cdot \om +l\cdot\Omega \big|\geq 2\<l\>-2\pi|k|\geq \<l\>\geq {\alpha} \  \frac{ \<l\>}{1+|k|^{\tau}}$.
 Thus Assumptions \ref{AS1} holds true. Further as the normal frequencies $\Omega_j=2j-1$ are constant, \ref{AS2} is satisfied.  \\[2pt]
 We now show that Assumption  \ref{AS3} holds. Because of the assumptions on the smoothness of $V$, the only condition which needs some care is that $(\frac{\partial Q}{\partial z_{k}})_{k\geq 1} \in \ell^{2}_{p}$. We have 
 \begin{equation*}
 \frac{\partial Q}{\partial z_{k}}=\int_{\R}V(\theta,x)h_{k}\ov{u} \,\text{d}x,
 \end{equation*}
which is  the $k$th coefficient of the decomposition of $V(\theta,x)\ov{u}$ in the Hermite basis. Thus $(\frac{\partial Q}{\partial z_{k}})_{k\geq 1} \in \ell^{2}_{2}$ if and only if $V(\theta,x)\ov{u} \in \H^{2}$ which is true since $\bar u\in \H^{2}$ and $V$ and $\partial_x V$ are bounded. \\
We turn to Assumption \ref{AS4}. Recall that by \eqref{YZ}, for all $2<r\leq +\infty$, there exists $\beta>0$ so that $\|h_{j}\|_{L^{r}(\R)}\leq C j^{-\beta}$. On the other hand, by assumption $V$ is  real analytic in $\theta$  and $L^q$ in $x$ for some $1\leq q<+\infty$. Consider $1<\ov{q}\leq +\infty$ so that $\frac1q+\frac1{\ov{q}}=1$,  then with H\"older, we  compute 
  \begin{eqnarray*}
\Big|\frac{\partial Q}{\partial z_{k}}\Big|=\big|\int_{\R}V(\theta,x)h_{k}\ov{u} \,\text{d}x\big|&\leq &\sup_{\theta\in [0,2\pi]^{n}}\|V(\theta,\cdot)\|_{L^{q}(\R)}\|h_{k}\,u\|_{L^{\ov{q}}(\R)}\\
&\leq &\sup_{\theta\in [0,2\pi]^{n}}\|V(\theta,\cdot)\|_{L^{q}(\R)}\|h_{k}\|_{L^{2\ov{q}}(\R)}\|u\|_{L^{2\ov{q}}(\R)}\\
&\leq &C k^{-\beta}.
\end{eqnarray*}
Similarly,
  \begin{eqnarray*}
\Big|\frac{\partial^{2} Q}{\partial z_{k}\partial \ov{z}_{l}}\Big|=\big|\int_{\R}V(\theta,x)h_{k}h_{l}\,\text{d}x\big|&\leq& \sup_{\theta\in [0,2\pi]^{n}}\|V(\theta,\cdot)\|_{L^{q}(\R)}\|h_{k}\|_{L^{2\ov{q}}(\R)}\|h_{l}\|_{L^{2\ov{q}}(\R)}\\
&\leq &C (jl)^{-\beta}. 
\end{eqnarray*}
 Therefore, Theorem \ref{thmKAM} applies (with $p=2$) and we almost obtain the conclusions of Theorem \ref{thmKAM2}.  Indeed, comparing with Theorem \ref{thmKAM},  we have to prove:
 \begin{itemize}
 \item[(i)] the symplectic coordinate transformation $\Phi$ is quadratic (and thus it is defined on  the whole phase space) and have the specific form \eqref{specific} ;
 \item[(ii)] the new normal form still have the same  frequencies vector $\om$ ;
 \item[(iii)] the new Hamiltonian reduces to the new normal form, i.e. $R^\star=0$ ;
 \item[(iv)] the symplectic coordinate transformation $\Phi$, which is defined by Theorem \ref{thmKAM2} on each  $\mathcal{P}^{2}$, extends to $\mathcal{P}^{0}=\T^{n}\times \R^{n}\times \ell_{2}^2\times\ell_{2}^2$.
 \end{itemize}
Actually, at the principle $Q$ is homogeneous of degree 2 in $Z$ and independent of $y$ and the same is true for $F$ the solution of the first homological equation 
$$\{F,N\}+\widehat N=\eps Q.$$
As a first consequence, $\widehat N$ does not contain linear terms in $y$ and thus $\omega$ remains unchanged by the first iterative step (cf. \eqref{newfreq}). Now going to Lemma \ref{lem.3} we notice that following notations \eqref{struct}, $b_0=b_1=a=0$. Therefore  $\theta$ remains unchanged ($\dot \theta=0$) and the equation for $Z$ reads $\dot Z= JA(\theta)Z$   which leads to $Z(\tau)=e^{\tau JA(\theta)}Z(0)$ (see \eqref{Zt}). Thus $Z(1)=L^{(1)}_\omega(\theta) Z(0)$ where $L^{(1)}_\omega(\theta)=e^{JA(\theta)}$ is invertible from $\mathcal{P}^{2}$ onto itself.\\
 In the same way, $\dot y(\tau)=-\frac 1 2 \nabla_\theta A(\theta)Z(\tau)\cdot Z(\tau)$ (see \eqref{eq.y}) which leads to $y(1)=y(0)+ \frac 1 2 Z(0)\cdot  M_\om(\theta)Z(0)$ for some linear operator $M_\om(\theta)$. Finally the new error term (cf. \eqref{newP}) $\dis Q_+=\int_0^1\{Q(t),F\}\circ X^t_F \,\text{d} t$ is still homogeneous of degree 2 in $Z$ and independent of $y$. Thus   properties (i), (ii) are satisfied after the first step and the new error term conserves the same form. Therefore we can iterate the process and the limiting transformation $\Phi=\Phi^1\circ \Phi^2\circ \cdots$ also satisfies (i) and (ii). Furthermore the transformed Hamiltonian as well as the original one is linear in $y$ and quadratic in $Z$ and thus (iii) holds true.\\
 It remains to check (iv). This follows from the fact that $\Phi$ is a {\it linear} symplectomorphism and thus, as remarked in \cite[Proposition 1.3']{Kuk3}, extends by duality on $\ell^2_p\times\ell^2_p$ for all $p\in [-2,2]$ and in particular for $p=0$.   \end{proof}
 \begin{proof}[Proof of Corollary \ref{Coro1.3}] {\it } The point is that, when $V$ is smooth with bounded derivatives, the perturbation $Q$ satisfies Assumption 3 for all $p\geq 0$. That is $X_Q$ maps smoothly $\mathcal P^{p}$ into itself. Therefore Theorem \ref{thmKAM} applies for all $p\geq 2$ and by \eqref{Est.Phi}, the canonical transformation $\Phi$ is close to the identity in the $\mathcal P^{p}$-norm. Since in the new variables, $(y',\theta',z',\bar z')=\Phi^{-1}(y,\theta,z,\bar z)$, the modulus of $z'_j$ is invariant, we deduce that there exist a constant $C$ such that
$$(1-C\eps) \|z(0)\|_{p}\leq \|z(t)\|_{p}\leq (1+C\eps) \|z(0)\|_{p}$$
which in turn implies
$$
(1-\eps C)\|u_{0}\|_{\H^{p}}\leq \|u(t)\|_{\H^{p}}\leq  (1+\eps C)\|u_{0}\|_{\H^{p}}, \quad \forall \,t\in \R.$$
 \end{proof}

\subsection{An explicit example}~\\[5pt]
 Consider the linear equation 
 \begin{equation}\label{lin.bis}
 i\partial_t u=-\partial^{2}_{x} u +x^2u +\epsilon V(t\omega)u
 \end{equation}
 where  $V:\T^{n}\longrightarrow \R$ is real analytic and  independent of $x\in \R$. Up to a translation of the spectrum, we can assume that   $V(0)=0$. Notice that this case is not in the scope  of Theorem \ref{thmKAM2}, since $V$ does not satisfy \eqref{*}.\\
 We suppose moreover that $\int_{\T^n}V=0$ and that $\om \in [0,2\pi)^{n}$ is Diophantine (see \eqref{diophantine}).
 
Define $v(t,x)=\e^{-i\eps \int_{0}^{t}V(\om s)\text{d}s}u(t,x)$. The function  $u$ satisfies \eqref{lin.bis} iff $v$ satisfies $ i\partial_t v=-\partial^{2}_{x} v +x^2v$. This latter equation is explicitly solvable using the Hermite basis, and the solution of \eqref{lin.bis} with initial condition  $u_{0}(x)=\sum_{j=1}^{\infty}\a_{j}h_{j}(x)$ then reads 
\begin{equation*} 
u(t,x)=\e^{i\eps \int_{0}^{t}V(\om s)\text{d}s}\sum_{j=1}^{\infty}\a_{j}h_{j}(x)\e^{i(2j-1)t}.
\end{equation*}
 Write $\dis V(\theta)=\sum_{k\in \Z^{n},k\neq 0} a_{k}\e^{ik\cdot \theta}$. Then, as $\om$ is Diophantine, we can compute 
$ \dis \int_{0}^{t}V(\om s)\text{d}s=-i\sum_{k\in \Z^{n},k\neq 0} \frac{a_{k}}{k\cdot \om}(\e^{ik\cdot \om t }-1)$,
  and $W$ defined by $\dis W(\theta)=\exp\big( \eps \sum_{k\in \Z^{n},k\neq 0} \frac{a_{k}}{k\cdot \om}(\e^{ik\cdot \theta }-1)    \big)$ is a periodic and  analytic function in $\theta$. Finally, $\dis u(t,x)=\sum_{j=1}^{\infty}\a_{j}W(\om  t)h_{j}(x)\e^{i(2j-1)t} $ is an almost periodic function in time (as an infinite sum of quasi-periodic functions).\\
 
We   can explicitly compute the transformation $\Phi$ in \eqref{specific}. Here the Hamiltonian reads $H=N+Q$ with $Q=V(\theta)\sum_{k\geq 1}|z_{k}|^{2}$. Set $\Phi(y',\theta',z',\ov{z}')=(y,\theta,z,\ov{z})$ where 
\begin{equation*} 
\left\{ \begin{array}{ll}     
\dis z_j=\dis W(\theta)\,z'_{j},\quad  {\bar z}_j=\ov{W(\theta)}\,{\bar z}'_{j}, & j\geq 1 \\[3pt]
 \dis \theta_j= \theta'_j,\quad  y_j=y'_{j}-\eps k_{j}\sum_{k\in \Z^{n},k\neq 0} \frac{a_{k}}{k\cdot \om}\e^{ik\cdot \theta }\sum_{l\geq 1}|z_{l}|^{2},   & 1\leq j\leq n.
  \end{array}\right.
\end{equation*} 
Then a straightforward computation gives 
$$H\circ \Phi(y',\theta',z',\ov{z}')=\sum_{j=1}^n \omega_j y'_j +\sum_{j\geq 1} (2j-1) z'_j\bar z'_j.$$
Therefore in this case $\Omega^{\star}_{j}(\om)=2j-1$.\\

 Finally we study the spectrum of the Floquet operator associated to the equation \eqref{lin.bis}. Observe that  $W(\om t)h_{j}(x)\e^{i(2j-1)t}$ solves \eqref{lin.bis} if and only if any $2j-1+k\cdot \om$ (with $j\geq 1$ and $k\in \Z^{n}$) is an eigenvalue of \eqref{Floq} (with eigenfunction $W(\theta)h_{j}(x)\e^{i\theta \cdot k}$). This shows that the Floquet spectrum is pure point, since  linear combinations of $W(\theta)h_{j}(x)\e^{i\theta \cdot k}$ are dense in $L^2(\R)\otimes L^2(\T^n)$.

\appendix
\section{Appendix}\label{appendix}
We show here how we can construct periodic solutions to the equation 
\begin{equation}\label{nls1}
\left\{
\begin{aligned}
&i\partial_t u+\partial^{2}_{x} u -x^{2}u =  |u|^{p-1}u ,\quad p\geq 1\quad 
(t,x)\in\R\times {\R},\\
&u(0,x)= f(x),
\end{aligned}
\right.
\end{equation}
thanks to variational methods. This is classical, see e.g. \cite{Struwe} and \cite{BereLions} for more details. See also \cite{Carles}.  Recall that for $s\geq 0$ we have defined the Sobolev space $\H^{s}(\R)=D(T^{s/2})$, where $T=-\partial^{2}_{x}+x^{2}$ is the harmonic oscillator. We also define $\H^{\infty}(\R)=\cap_{s>0}\H^{s}(\R)$. We then have the following result.

\begin{prop}
Let $\mu>0$. Then there exists an $L^{2}(\R)$-orthogonal family $(\phi^{j})_{j\geq 1}\in \mathcal{H}^{\infty}(\R)$ with $\|\phi^{j}\|_{L^{2}(\R)}=\mu$ and a sequence of positives numbers $(\lambda_{j})_{j\geq 1}$ so that for all $j\geq 1$, $u(t,x)=\e^{-i\lambda_{j}t}\phi^{j}(x)$ is a solution of \eqref{nls1}.
\end{prop}

\begin{proof}
We look for a solution of \eqref{nls1} of the form $\dis u(t,x)=\e^{-i\lambda t}\phi(x) $, hence $\phi$ has to satisfy 
\begin{equation}\label{Elliptique}
\big(-\partial_{x}^{2}+x^{2}\big)\phi=\lambda \phi-|\phi|^{p-1}\phi.
\end{equation}
Let $\mu>0$, denote by $E_{\mu}$ the set 
\begin{equation*}
E_{\mu}=\Big\{ \phi\in \H^{1}(\R),\;s.t.\;\|\phi\|_{L^{2}(\R)}=\mu  \Big\},
\end{equation*}
and define the functional 
\begin{equation*}
J(\phi)=\int\frac12\big((\partial_{x} \phi)^{2}+x^{2}\phi^{2}\big)+\frac{1}{p+1}|\phi|^{p+1}\big)\text{d}x,
\end{equation*}
which is $\mathcal{C}^{1}$ on $E_{\mu}$.\\
Then the problem $\dis \min_{\phi\in E_{\mu}} J(\phi)$ admits a  solution $\phi^{1}$, and $\phi^{1}$ solves \eqref{Elliptique} for some $\lambda=\lambda_{1}>0$.\\
Indeed, by Rellich's theorem (see e.g. \cite[page 247]{ReedSimon4}), for all $C>0$, the set 
\begin{multline*}
 \Big\{ \phi\in \H^{1}(\R),\;s.t.\;\|\phi\|_{L^{2}(\R)}=\mu,\\
 \int\frac12\big((\partial_{x} \phi)^{2}+x^{2}\phi^{2}\big)+\frac{1}{p+1}|\phi|^{p+1}\leq C  \Big\},
 \end{multline*}
is compact in $L^{2}(\R)$ (observe that we have used the Sobolev embedding $\H^{1}\subset L^{p+1}$ which holds for any $p\geq 1$). Then, if   $\phi_{n}$ is a minimising sequence of $J$, up to a sub-sequence, we can assume that $\phi_{n}\longrightarrow \phi^{1} \in E_{\mu}$ in $L^{2}(\R)$. Finally, the lower  semicontinuity of $J$  ensures that $\phi^{1}$ is a minimum of $J$ in $E_{\mu}$, and the claim follows. Moreover, $\lambda_{1}$ is given by 
\begin{equation*}
\lambda_{1}=\frac{1}{\mu}\int(\partial_{x} \phi^{1})^{2}+x^{2}(\phi^{1})^{2}+|\phi^{1}|^{p+1}.
\end{equation*}
Now we define the set $\dis E_{\mu}^{1}=E_{\mu}\cap \big\{ \<\phi,\phi^{1}\>_{L^{2}(\R)}=0\big\}$. Similarly, we may construct $\phi^{2}\in E^{1}_{\mu}$ so that $\dis J (\phi^{2})=\min_{\phi \in E^{1}_{\mu}}J(\phi)$. The orthogonality condition implies in particular that $\phi^{2}\neq \phi^{1}$. Let $k\geq 1$, and assume that we have constructed $\dis (\phi^{j})_{1\leq j \leq  k}$ so that $\<\phi^{i},\phi^{j}\>_{L^{2}}=\mu^{2} \delta_{ij}$ for all $1\leq i,j\leq k$. Define the set 
\begin{equation*}
E^{k}_{\mu}=E_{\mu}\cap \big\{ \<\phi, \phi^{j}\>_{L^{2}}=0,\;1\leq j\leq k\big\}.
\end{equation*}
By Rellich's theorem, the set 
\begin{multline*}
 \Big\{ \phi\in \H^{1}(\R),\;s.t.\;\|\phi\|_{L^{2}(\R)}=\mu,\\
 \int\frac12\big((\partial_{x} \phi)^{2}+x^{2}\phi^{2}\big)+\frac{1}{p+1}|\phi|^{p+1}\leq C, \;\;  \<\phi, \phi^{j}\>_{L^{2}}=0,\;1\leq j\leq k\Big\},
 \end{multline*}
is compact in $L^{2}(\R)$ and we can construct $\phi^{k+1}\in E^{k}_{\mu}$ so that $\dis J (\phi^{k+1})=\min_{\phi \in E^{k}_{\mu}}J(\phi)$. Then $\phi^{k+1}$ is a nontrivial solution of \eqref{Elliptique} with 
\begin{equation*}
\lambda_{k+1}=\frac{1}{\mu}\int(\partial_{x} \phi^{k+1})^{2}+x^{2}(\phi^{k+1})^{2}+|\phi^{k+1}|^{p+1}.
\end{equation*}
The regularity $\phi^{j}\in \mathcal{H}^{\infty}$ is a direct consequence of the ellipticity of the operator $-\partial_{x}^{2}+x^{2}$.
  \end{proof}
  
  \begin{rema}
  Of course, the proof can be generalised to a larger class of nonlinearities in \eqref{nls1}. In particular, we can deal with the nonlinearity $-\eps |u|^{p-1}u$ with $\eps>0$ provided that  $p<5$ and that $\eps \mu^{\frac{p+3}2}>0$ is small enough. Indeed in that case, thanks to the Gagliardo-Nirenberg inequality we have 
  $$ \frac{\eps}{p+1}\int |\phi|^{p+1} \leq C \eps \mu^{\frac{p+3}2}  \Big(\int(\partial_{x} \phi)^{2}+x^{2}\phi^{2}\Big)^{(p-1)/4},
  $$
  and the nonlinear part of the energy can be controlled by the linear part, which enables us the perform the same arguments as previously.
 \end{rema}

\end{document}